\documentclass[11pt]{amsart}
\usepackage{amssymb, latexsym}
\usepackage{graphicx}
\theoremstyle{plain}
\newtheorem{theorem}{Theorem}
\newtheorem{corollary}{Corollary}

\newtheorem {lemma}{Lemma}
\newtheorem{proposition}{Proposition}

\theoremstyle{remark}
\newtheorem*{remark}{Remark}
\newtheorem*{Remark 1}{Remark 1}
\newtheorem*{Remark 2}{Remark 2}
\newtheorem*{Remark 3}{Remark 3}
\newtheorem*{Remark 4}{Remark 4}

\numberwithin{equation}{section}

\begin{document}

\title[Card-Cyclic to Random Shuffle]%
 {Probabilistic and  Combinatorial Aspects of the Card-Cyclic to Random Insertion Shuffle}

\author{Ross G. Pinsky}
\address{Department of Mathematics\\
Technion---Israel Institute of Technology\\
Haifa, 32000\\ Israel}
\email{ pinsky@math.technion.ac.il}
\urladdr{http://www.math.technion.ac.il/~pinsky/}
\thanks{}

\date{}

\begin{abstract}
Consider  a permutation  $\sigma\in S_n$ as a deck of cards numbered from 1 to $n$ and laid out in a row,
 where
$\sigma_j$ denotes the number of the card that is in the $j$-th
position from the left.\rm\ We study some probabilistic and combinatorial aspects of
the  shuffle on $S_n$ defined   by
removing and then randomly reinserting each of the $n$ cards once,
with the removal and reinsertion being performed according to the original left to right order
of the cards. The novelty here in this nonstandard shuffle is that every card is removed and reinserted exactly once.
The bias that remains turns out to be quite strong and possesses some surprising features.

\end{abstract}

\maketitle

\section{Introduction and Statement of Results}
Let $S_n$ denote  the symmetric group of permutations of $[n]\equiv\{1,\cdots, n\}$.
Our convention will be to  view  a permutation $\sigma\in S_n$ as a deck of cards numbered from 1 to $n$
and laid out in a row, where
$\sigma_j$ denotes the number of the card that is in the $j$-th position from the left.\rm\
In this paper, we analyze the bias in the following  ``shuffle'' on $n$ cards:
remove and then randomly reinsert each of the $n$ cards exactly once,
the removal and reinsertion  being performed according to the \it original\rm\ left to right order
of the cards.
The novelty here in this nonstandard shuffle  is that {\it every card is removed and reinserted exactly once}, unlike in any of the shuffles
one encounters in the literature. The point is to see how much bias remains when one knows that every card has been
removed and reinserted.

We dub this shuffle the  \it card-cyclic to random insertion shuffle\rm. The reason for this terminology along with
the original motivation that led to the study of this shuffle will be explained at the end of this section.
However, we feel that the results are of independent interest regardless of that motivation.

  We let $p_n(\sigma,\tau)$ denote the probability that the deck ends up
in the state $\tau\in S_n$, given that it began in state $\sigma\in S_n$. Of course, since the shuffle is transitive,
it suffices to look at $p_n(\text{id},\cdot)$, where id is the identity element, corresponding to the cards being
in increasing order from left to right.
Note that if $n\ge3$,
 the distribution after one such shuffle
can not be exactly uniform  because there are $n^n$ equally probable ways to implement the shuffle, but there are
$n!$ possible states of the deck, and $n!\nmid n^n$.
Of course this doesn't rule out asymptotic uniformity, but in fact we shall see that the
card-cyclic to random insertion shuffle is far from uniform.

We begin with the behavior of the distribution of the card in the first position and of the card in the last position.
The bias with regard to the first position turns out to be  quite strong.
\begin{theorem}\label{First}
Under $p_n(\text{id},\cdot)$, the random variable $\sigma_1$, denoting the number of the card in the first position, has the following
behavior:

\noindent i.
\begin{equation}
\lim_{n\to\infty}np_n(\text{id},\{\sigma_1=b_nn\})=e^{b-1}, \ \text{if}\ \ \lim_{n\to\infty}b_n=b\in(0,1].
\end{equation}
In particular then, defining the  probability measure $\nu^{\text{F}}_n$ on $[0,1]$ by
$$
\nu^{\text{F}}_n(A)=p_n(\text{id},\{\sigma_1\in nA\}), A\subset [0,1],
$$
one has
$$
w-\-lim_{n\to\infty}\nu^{\text{F}}_n(dx)=e^{-1}\delta_0(x)+e^{x-1}dx.
$$
\noindent ii
\begin{equation}
\begin{aligned}
&\lim_{n\to\infty}np_n(\text{id},\{\sigma_1=b_nn\})=e^{-1},\\
& \text{if}\ \lim_{n\to\infty} b_n=0\ \text{and}\
\liminf_{n\to\infty}\frac{n^{\frac12}}{\sqrt{\log n}}b_n>\sqrt2.
\end{aligned}
\end{equation}
\noindent iii.
\begin{equation}
\begin{aligned}
&\lim_{n\to\infty}n^\frac12p_n(\text{id},\{\sigma_1=d_nn^\frac12\})=e^{-1}\int_d^\infty e^{-\frac{y^2}2}dy,\\
&\text{if}\
 \lim_{n\to\infty}d_n=d\in[0,\infty].
\end{aligned}
\end{equation}
In particular then, defining the  sub-probability measure $\mu^{\text{F}}_n$ on $[0,\infty)$ by
$$
\mu^{\text{F}}_n(A)=p_n(\text{id},\{\sigma_1\in n^\frac12A\}), A\subset [0,\infty),
$$
one has
$$
w-\lim_{n\to\infty}\mu^{\text{F}}_n(dx)=e^{-1}\big(\int_x^\infty e^{-\frac{y^2}2}dy\big)dx,
$$
the total mass of the measure on the right hand side above being $e^{-1}$.
\end{theorem}

 \bf\noindent Remark 1.\rm\ From Theorem \ref{First}, it follows that \it the most
 likely numbers for the first position lie ``right next to'' the least likely numbers.\rm\ More precisely,
  the following facts  follow from Theorem \ref{First}:

\noindent 1. (Most likely asymptotic numbers for first position) Let $\{\gamma_n\}_{n=1}^\infty$ denote a sequence satisfying  $1\le \gamma_n\le n$, for each $n$. Then
$$
\sup_{\{\gamma_n\}_{n=1}^\infty}\limsup_{n\to\infty}n^\frac12p_n(\text{id},\{\sigma_1=\gamma_n\})=
\frac{\sqrt{2\pi}}{2e}.
$$
In particular, the supremum is attained for sequences $\{\gamma_n\}_{n=1}^\infty$ satisfying
$\gamma_n=o(n^\frac12)$.

\noindent 2. (Least likely asymptotic numbers for first position) Let $\{\gamma_n\}_{n=1}^\infty$ denote a sequence satisfying  $1\le \gamma_n\le n$, for each $n$. Then
$$
\inf_{\{\gamma_n\}_{n=1}^\infty}\liminf_{n\to\infty}np_n(\text{id},\{\sigma_1=\gamma_n\})=e^{-1}.
$$
In particular, the infimum is attained for sequences $\{\gamma_n\}_{n=1}^\infty$ satisfying $\gamma_n=o(n)$
and $\gamma_n\ge(\sqrt2+\epsilon)n^\frac12\sqrt{\log n}$,
for some $\epsilon>0$.

\noindent \bf Remark 2.\rm\ Note that the boundary layer between
$p_n(\text{id},\{\sigma_1=j\})$ being on the order
$n^{-\frac12}$ and being on the order $n^{-1}$ is the narrow
strip   where $j$ is on a larger order than $n^\frac12$
but on an order no larger than $n^\frac12\log n$.
\medskip

The bias
with regard to the last position is  considerably tamer than the bias with regard to the first position.
\begin{theorem}\label{Last}
Under $p_n(\text{id},\cdot)$, the random variable $\sigma_n$, denoting the number of the card in the last position, has the following
behavior:

\noindent i.
\begin{equation}
\lim_{n\to\infty}np_n(\text{id},\{\sigma_n=b_nn\})=\frac{e^b}{e-1}, \
\text{if}\ \lim_{n\to\infty}b_n=b\in[0,1).
\end{equation}
In particular then, defining the  probability measure $\nu^{\text{L}}_n$ on $[0,1]$ by
$$
\nu^{\text{L}}_n(A)=p_n(\text{id},\{\sigma_n\in nA\}), A\subset [0,1],
$$
one has
$$
w-\-lim_{n\to\infty}\nu^{\text{F}}_n(dx)=\frac{e^x}{e-1}dx.
$$
\noindent ii.
\begin{equation}
\lim_{n\to\infty}np_n(\text{id},\{\sigma_n=b_nn\})=\frac e{e-1}, \ \text{if}\ \ \lim_{n\to\infty}b_n=1
\ \text{and}\ \lim_{n\to\infty}(n-b_nn)=\infty;
\end{equation}
\noindent iii.
\begin{equation}
\lim_{n\to\infty}np_n(\text{id},\{\sigma_n=n-l\})=\frac{e-e^{-l}}{e-1}, \ l=0,1,\cdots.
\end{equation}
\end{theorem}

\noindent \bf Remark.\rm\
The following facts  follow from Theorem \ref{Last}:

\noindent 1. (Most likely asymptotic numbers for last position)
Let $\{\gamma_n\}_{n=1}^\infty$ denote a sequence satisfying  $1\le \gamma_n\le n$, for each $n$.
Then
$$
\sup_{\{\gamma_n\}_{n=1}^\infty}\limsup_{n\to\infty}np_n(\text{id},\{\sigma_n=\gamma_n\})=\frac e{e-1}.
$$
In particular, the supremum is attained for sequences $\{\gamma_n\}_{n=1}^\infty$
satisfying $\lim_{n\to\infty}\frac{\gamma_n}n=1$ and $\lim_{n\to\infty}(n-\gamma_n)=\infty$.

\noindent 2. (Least likely asymptotic numbers for last position)
Let $\{\gamma_n\}_{n=1}^\infty$ denote a sequence satisfying  $1\le \gamma_n\le n$, for each $n$.
Then
$$
\inf_{\{\gamma_n\}_{n=1}^\infty}\liminf_{n\to\infty}np_n(\text{id},\{\sigma_n=\gamma_n\})=\frac 1{e-1}.
$$
In particular, the supremum is attained for sequences $\{\gamma_n\}_{n=1}^\infty$
satisfying $\gamma_n=o(n)$.

\medskip

Theorem \ref{First} showed that the cards with numbers on the order $n^\frac12$ are more likely to occupy the first
position than cards with larger numbers. In fact, more generally, cards with numbers on the order $n^\frac12$ are more
likely to occupy any position at the beginning of the deck  than are cards with larger numbers.
We can quantify this and use it to prove that the total variation norm between the card-cyclic to random insertion shuffle measure
 and the uniform measure
converges to 1 as $n\to\infty$.
Recall that the total variation norm between two probability measures $\mu$ and $\nu$ on $S_n$ is defined
by
$$
||\mu-\nu||_{\text{TV}}=\sup_{A\subset S_n}(\mu(A)-\nu(A))=\frac12\sum_{\sigma\in S_n}|\mu(\sigma)-\nu(\sigma)|.
$$
\begin{theorem}\label{TV}
Let
\begin{equation*}
A^{(n)}_{M,L}=\{\sigma\in S_n:\sigma_j\le Mn^\frac12,\ \text{for some}\ j\le L\}
\end{equation*}
be the event that a card with  a number less than or equal to $Mn^\frac12$ appears in one of the first $L$ positions.
Then for sufficiently large $C$,
\begin{equation}\label{TVcond}
\lim_{M\to\infty}\lim_{n\to\infty}p_n(\text{id},A^{(n)}_{M,CM^2})=1.
\end{equation}
In particular then,
\begin{equation}\label{TV0}
\lim_{n\to\infty}||p_n(\text{id},\cdot)-U_n||_{\text{TV}}=1.
\end{equation}
\end{theorem}

\medskip

The first two theorems dealt with the distribution of the {\it number of the  card} in special positions---namely, the first and the last positions.
We now consider the distribution of the {\it position of the   card} with a general  number.
\begin{theorem}\label{general}
Under $p_n(\text{id},\cdot)$, the random variable $\sigma^{-1}_{b_nn}$, denoting the position of card number
$b_nn$,     has the following behavior.
Assume that $\lim_{n\to\infty}b_n=b\in[0,1]$.
Then the weak limit of the distribution of $\frac1n\sigma^{-1}_{b_nn}$ exists. Its distribution function
$$
F_b(x)\equiv\lim_{n\to\infty}p_n(\text{id},\{\sigma^{-1}_{b_nn}\le xn\}), \ x\in[0,1],
$$
is given as follows. Define
$G_b:[0,1]\to[0,1]$ by
$$
G_b(y)=\begin{cases}
 ye^{1-b},\ & 0\le y\le 1-(1-b)e^b;\\
 e^{(1-y)e^{-b}}-(1-y)e^{1-b}, \ &  1-(1-b)e^b\le y\le 1.
\end{cases}
$$
Then $F_b=G_b^{-1}$.
\end{theorem}
A calculus exercise gives the following corollary of the theorem.
\begin{corollary}\label{density}
Let $f_b$ denote the density of  the distribution function $F_b$, that is, the density function for the limiting
rescaled position of a  card with a number around $bn$. Let
$x_b=e^{1-b}-(1-b)e$.
Then $f_1\equiv1$, and for $b\in[0,1)$,

\noindent i. $f_b(x)=e^{b-1}, \ 0\le x<x_b$;

\noindent ii. $f_b(x_b^+)=\frac{e^{b-1}}{1-e^{-b}}$;

\noindent iii. $f_b(x)$ is decreasing and convex for $x\in(x_b,1]$;

\noindent iv. $f_b(1)=\frac{e^{b-1}}{1-e^{-1}}$.
\end{corollary}

\noindent \bf Remark.\rm\ The corollary shows that for a card with a number around $bn$, with $b\in(0,1)$, the most likely
positions in which it will end up are those just to the right of $nx_b=n(e^{1-b}-(1-b)e)$, and the least likely
positions are all of those less than $nx_b$. In particular,  \it the  most likely positions for a card
with a number around $bn$ lie ``right next to'' the least likely positions.\rm\ See figure \ref{Fi:f_b}.
Note also that $f_0(0)=\infty$, which means that for all $b$ and $d$, the probability of a card with a number around
$bn$ ending up in a position around $dn$ is the greatest for $b=d=0$.
(This connects up with Theorem \ref{First}.)
The probability measures corresponding to the densities $f_b$ are weakly continuous with respect to $b\in[0,1]$.
For each $b\in(0,1)$, the density $f_b$ has a discontinuity, however considered as cadlag functions,
the densities $f_b$  vary continuously in the Skorohod topology for $b\in(0,1)$. This continuity does not extend
to $b=0$, where $f_0(0)=\infty$, or to $b=1$, where $f_1\equiv1$ but
$\lim_{b\to1}\sup_{x\in[0,1]}f_b(x)=\frac e{e-1}$.
\medskip

\begin{figure}
\includegraphics[scale=.65]{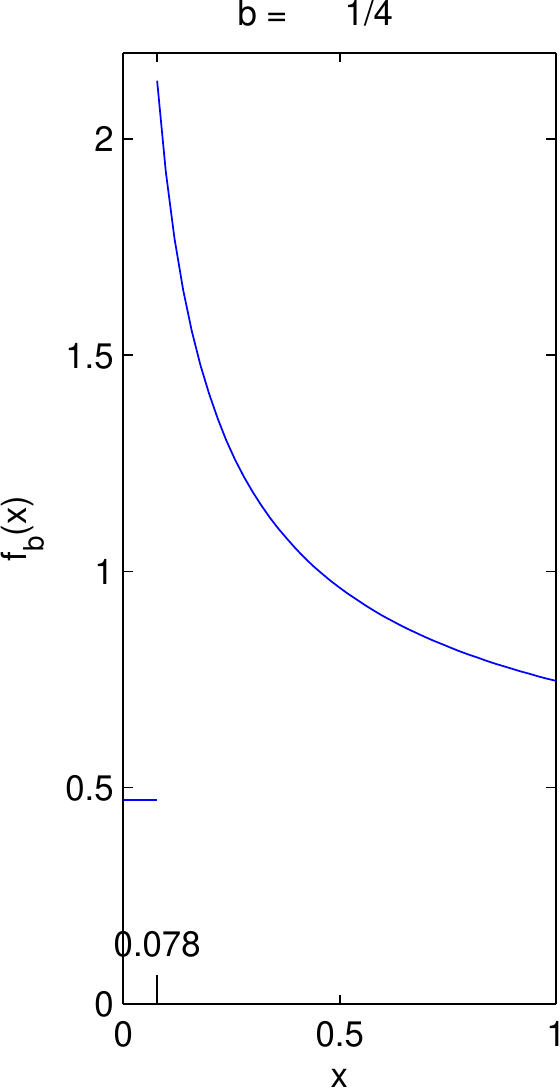}
\includegraphics[scale=.65]{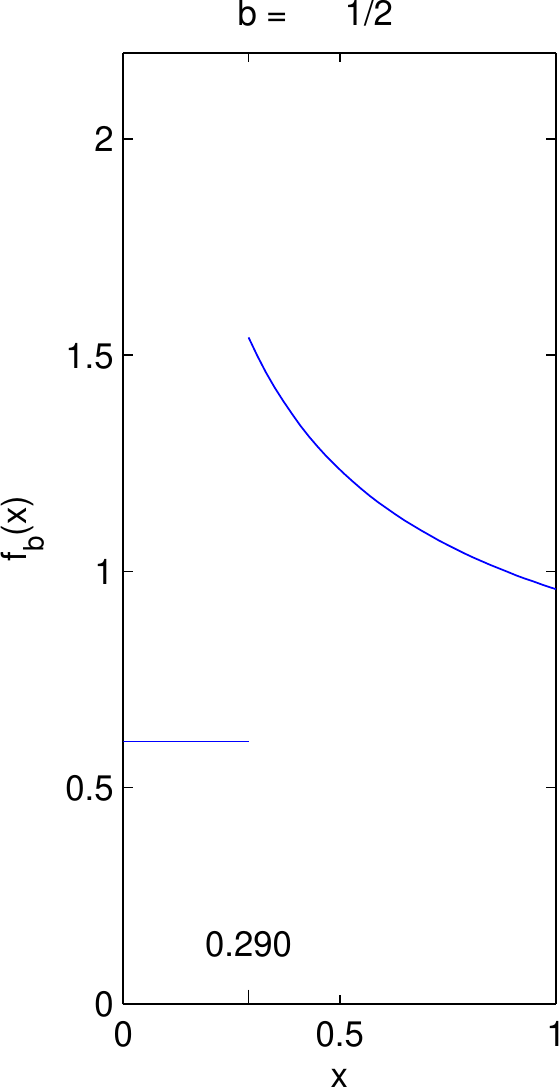}
\includegraphics[scale=.65]{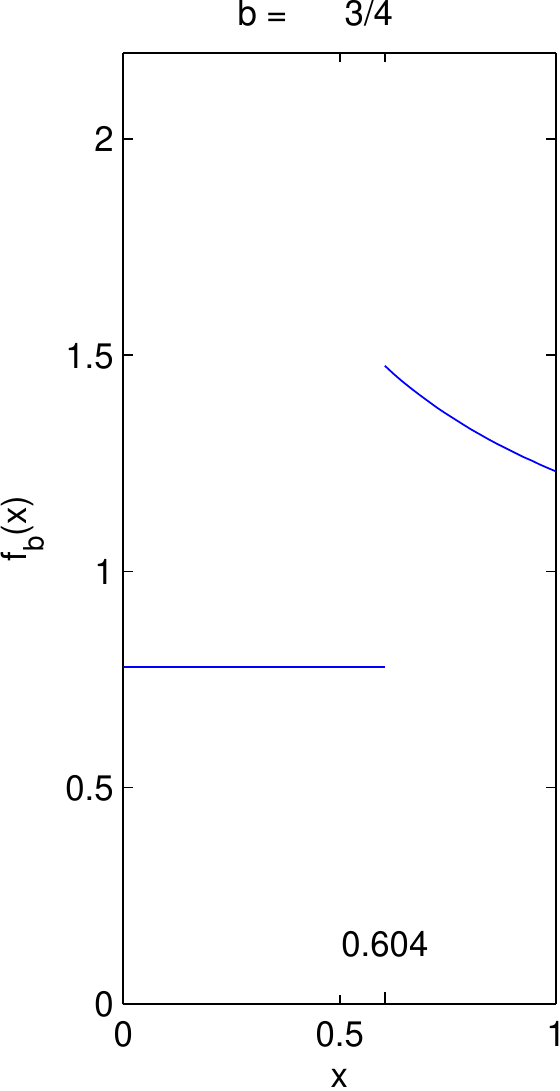}
\caption{Density for limiting rescaled position of a card with a number around $bn$.}\label{Fi:f_b}
\end{figure}

Let $E^{p_n(\text{id},\cdot)}$ denote the expectation corresponding to the card-cyclic to random insertion shuffle starting from
id, so that $E^{p_n(\text{id},\cdot)}\sigma^{-1}_{b_nn}$ is the expected position for card number $b_nn$ at the end of the shuffle.
It follows from the theorem that
if $\lim_{n\to\infty}b_n=b$, then $E(b)\equiv\lim_{n\to\infty}\frac1nE^{p_n(\text{id},\cdot)}\sigma^{-1}_{b_nn}$
exists and is given by $\int_0^1(1-F_b(x))dx$. Making a substitution and integrating by parts shows that
this integral is equal to $\int_0^1G_b(y)dy$. Computing this integral then gives the following corollary.

\begin{corollary}\label{expectation}
Let $E(b)\equiv\lim_{n\to\infty}\frac1nE^{p_n(\text{id},\cdot)}\sigma^{-1}_{b_nn}$, where $\lim_{n\to\infty}b_n=b$, denote the rescaled limiting expected position
for a card with a number around $bn$.
Then
$$
E(b)=eb+\frac12e^{1-b}-e^b.
$$
The function $E(b)$ has the following properties:

\noindent i. $E(0)=\frac12e-1\approx .359$;

\noindent ii. $E(1)=\frac12$;

\noindent iii. $E(\cdot)$ increases for $b\in[0,b^*]$ and decreases for $b\in[b^*,1]$, where
$b^*\approx .722$ is the solution to $e-e^b-\frac12e^{1-b}=0$.
The maximum value of $E(\cdot)$ is \newline $E(b^*)\approx .564$;

\noindent iv. $E(b)\ge b$, for $b\in[0,\bar b]$ and $E(b)\le b$ for $b\in[\bar b,1]$, where
$\bar b\approx .545$.

\noindent v. $\int_0^1E(b)db=\frac12$.

\end{corollary}

\bf\noindent Remark.\rm\ In particular, a card starting out very near the left end of the deck will end up on the average
around 35.9 percent of the way through  the deck, while a    card starting out anywhere else will end up on the average further  to the right
than this.
A card starting out  around 72.2 percent of the way through the deck will end up on the  average around 56.4 percent
of the way through the deck, while  a  card starting out anywhere else will  end up on the average further to the left than this.
A card in the first 54.5 percent of the deck  will end up on the average further to the right than where it started,
while  a card in the last 45.5 percent of the deck will  end up on the average further to the left than where it started. See figure \ref{Fi:E}.
But of course, as (v) indicates and as is clear from considerations of symmetry, the average ending position of the average card  must be the 50th percentile.
\medskip

\begin{figure}
\includegraphics[scale=.6]{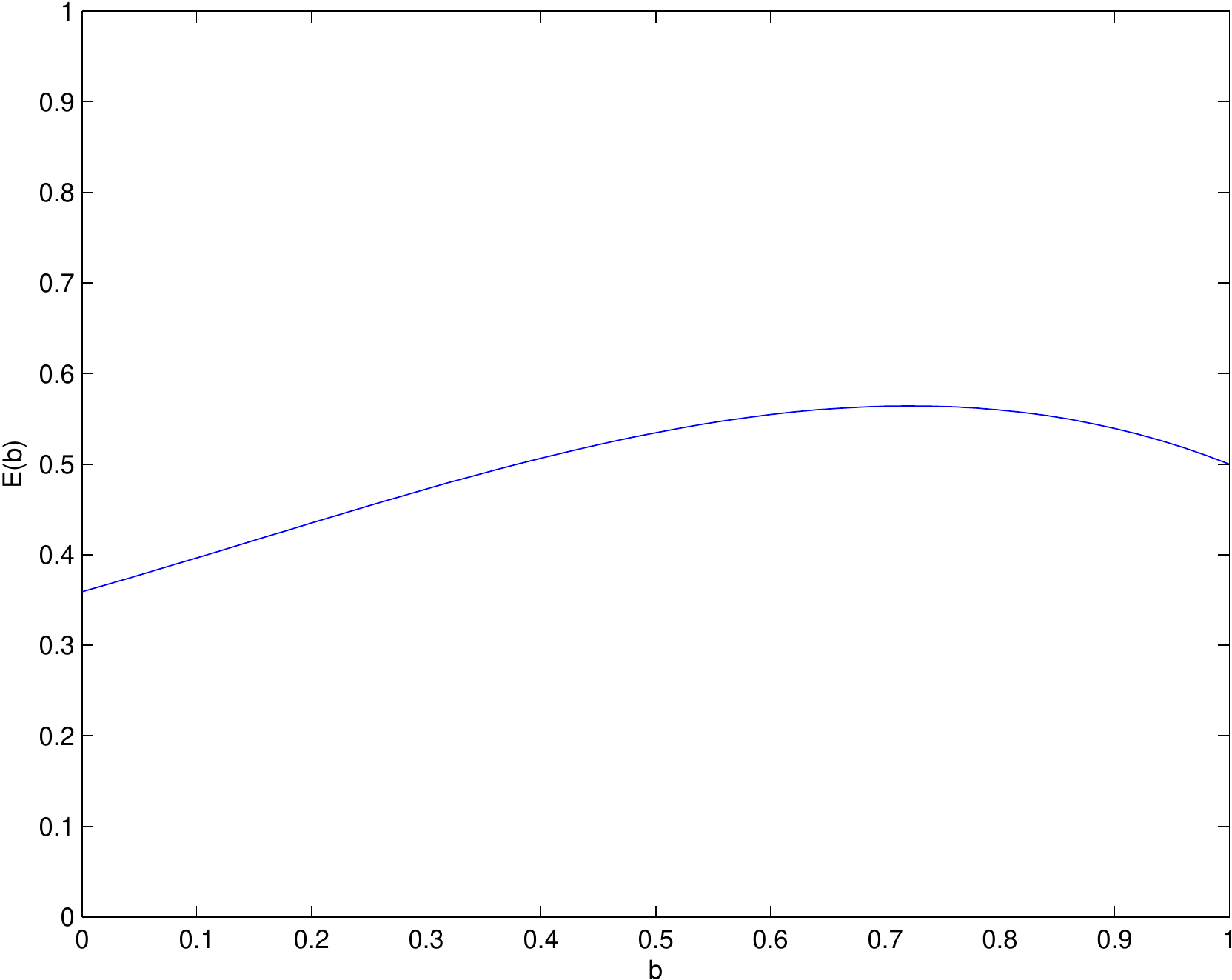}
\caption{Limiting rescaled  expected position of a card with a number around $bn$.}\label{Fi:E}
\end{figure}

The following corollary shows that the random positions
of  a finite number of cards are asymptotically independent.
The result  follows easily from the proof of Theorem \ref{general}, as will be shown after the proof of that theorem.
\begin{corollary}\label{joint}
For $m\ge 1$, let $1\le b_{1n}n<b_{2n}n\cdots< b_{mn}n\le n$ satisfy $\lim_{n\to\infty}b_{jn}=b_j\in[0,1]$;
so $0\le b_1\le b_2\le \cdots\le b_m\le 1$.
Then under $p_n(\text{id},\cdot)$, the   distribution of the random vector  $(\sigma^{-1}_{b_{1n}n},\sigma^{-1}_{b_{2n}n},\cdots,\sigma^{-1}_{b_{mn}n})$
converges to the $m$-dimensional product distribution with density $\prod_{j=1}^mf_{b_j}(x_j)$, where $x=(x_1,\cdots, x_m)$.
\end{corollary}

We can use the above corollary to say something about the probability of inversions.
For $i<j$, if card number $j$ appears to the left of card number $i$ in a permutation $\sigma$,
then we say that the pair of cards with numbers $i$ and $j$ form  an \it inversion\rm\ for  the  permutation $\sigma$.
This concept is described more fully below, two paragraphs above Lemma \ref{counting}.
For $m=2$ in  Corollary \ref{joint}, let $(\Sigma^{-1}_{1;b_1},\Sigma^{-1}_{2;b_2})$ denote
a random vector distributed according to the density $f_{b_1}(x_1)f_{b_2}(x_2)$.
We will prove the following result.
\begin{corollary}\label{pairs}

\noindent i. Let $\hat  b\approx.768$ be the root of the equation
$(1-b)e^b-\frac12=0$. Then
$$
\begin{aligned}
&P(\Sigma^{-1}_{1;b_1}<\Sigma^{-1}_{2;b_2})>\frac12,\ \text{for}\ b_1<b_2<\hat b\ \text{and}\ b_2\ \text{sufficiently close to}\ b_1;\\
&P(\Sigma^{-1}_{1;b_1}<\Sigma^{-1}_{2;b_2})<\frac12,\ \text{for}\ \hat b<b_1<b_2\ \text{and}\ b_2\ \text{sufficiently close to}\ b_1;
\end{aligned}
$$
\noindent ii. Let $\tilde b\approx.380$ be the unique root of the equation $E(b)=\frac12$, for $b\in[0,1)$,  where $E(b)$ is as in Corollary \ref{expectation}.
Then
$$
\begin{aligned}
&P(\Sigma^{-1}_{1;b}<\Sigma^{-1}_{2;1})>\frac12,\ \text{for}\ b\in[0,\tilde b);\\
&P(\Sigma^{-1}_{1;b}<\Sigma^{-1}_{2;1})<\frac12,\ \text{for}\ b\in(\tilde b,1).
\end{aligned}
$$
\end{corollary}
\bf \noindent Remark.\rm\
The first part of the  corollary indicates   that for large $n$,  if one takes a card
with a number around $b_1n$ and a card with a number around $b_2n$, with $b_2>b_1$ and sufficiently close to $b_1$,
then under $p_n(\text{id},\cdot)$, the probability that these cards form an inversion
is less than $\frac12$  if $b_1<\hat b\approx .768$ and greater than $\frac12$ if $b_1>\hat b\approx.768$.
 (We suspect that the restriction that $b_2$ be close to $b_1$ is unnecessary for the above result.)
The second part of the corollary indicates that for large $n$,
if one takes a card with a number around $bn$, $b\in(0,1)$,  and a card with a number around $n$ (that is, a card from the very end of the deck),  then
under $p_n(\text{id},\cdot)$, the probability that these  cards form an inversion
is  less than $\frac12$ if $b<\tilde b$ and greater than $\frac12$ if $b>\tilde b$. Furthermore, the  point $b=\tilde b$ where the probability
is equal to $\frac12$ is exactly the  point $b$ where the limiting average rescaled position $E(b)$ is equal to $\frac12$.
Despite the above corollary, the measure $p_n(\text{id},\cdot)$ favors permutations that do not have a lot of  inversions,
in  a sense made precise in Corollary \ref{inverting} below. See also, the remark after that corollary.

\medskip

The results in Theorems \ref{First} and \ref{Last} are local limit theorems. If we had such a local result in
Theorem \ref{general}; namely
$\lim_{n\to\infty}\frac1np_n(\text{id},\{\sigma^{-1}_{b_nn}= [xn]\})=f_b(x)$, rather than only
\begin{equation}\label{4again}
\lim_{n\to\infty}p_n(\text{id},\{\sigma^{-1}_{b_nn}\le xn\})=F_b(x),
\end{equation}
 then it would follow easily that
if $\lim_{n\to\infty}x_n=x$, then
$\lim_{n\to\infty}p_n(\text{id},\{\sigma_{x_nn}\le bn\})= \int_0^bf_b(x)db$.
Unfortunately, we don't see how to prove this rigorously just from \eqref{4again},
nor do we see how to prove directly that
$\lim_{n\to\infty}p_n(\text{id},\{\sigma_{x_nn}\le bn\})$ exists; although it is intuitively obvious
that it does. And if it does exist, then it is easy to show that the corresponding density
must be $h_x(b)\equiv f_b(x)$.
This density function $h_x(b)$, $0\le b\le 1,$ for the limiting rescaled expected card number occupying a position
around $xn$, is not  useful for explicit calculations as is  $f_b(x)$, the
density function for the limiting
rescaled position of a  card with a number around $bn$.  However, we can give its basic behavior, like
we gave the basic behavior of $f_b(x)$ in Corollary \ref{density}. We will prove the following result concerning the behavior of $h_x(b)\equiv f_b(x)$.
\begin{corollary}\label{densityatpos}
The density function $h_x(b)$, the limiting rescaled   card number occupying   a position around $xn$, has the following behavior.

For $x=0$,  one has  $h_0(b)=e^{b-1}, 0< b\le 1$.
This is a sub-probability density with total mass $1-e^{-1}$. In addition there is a $\delta$-mass of size $e^{-1}$ at $b=0$.

For $x=1$, one has $h_1(b)=\frac{e^b}{e-1}, 0\le b\le 1$.

  Let $b_x$, $0\le x< 1$, denote the inverse of the function $x_b=e^{1-b}-(1-b)e$.
For $0<x< 1$, one has

\noindent i. $h_x(0)=\frac{e^{b_x-1}}{e^{b_x}-1}$;

\noindent ii. $h_x(b)$ is increasing and convex on $0<b<b_x$;

\noindent iii. $h_x(b_x^-)=\frac{e^{b_x-1}}{1-e^{-b_x}}$;

\noindent iv. $h_x(b)=e^{b-1},\ b_x<b\le 1$.
\end{corollary}
\noindent \bf Remark.\rm\ The fact that at $x=0$ there is a $\delta$ mass at 0 of size $e^{-1}$ connects
up with Theorem \ref{First}. The corollary shows that the most likely numbers to find in a  position around $xn$, $0<x< 1$,
 are  numbers  slightly smaller than  $nb_x$. If $e^{b_x}<2$, or equivalently,  $x<\frac12 e-(1-\ln 2)e\approx.525$, then the   least likely numbers
to find in  a position around $xn$ are  numbers slightly larger than   $nb_x$; if $x>\frac12 e-(1-\ln 2)e$, then the least likely numbers
to find in  a position around $xn$ are  numbers  on order $o(n)$.
\it In particular,
for all $x\in[0,1]$, the most likely numbers for  a position around $xn$ are  ``right next to'' numbers that are much less likely to be in such a position,
and if $x<\frac12 e-(1-\ln 2)e\approx.525$, then these latter numbers are the least likely ones to be
in such a position.\rm\ See figure \ref{Fi:h_x}.
The probability  measures corresponding to the densities $h_x$ are weakly continuous with respect to $x\in(0,1]$,
and the densities $h_x$, considered as cadlag functions with the Skorohod topology, vary continuously
for $x\in(0,1]$.
\rm

\begin{figure}
\includegraphics[scale=.55]{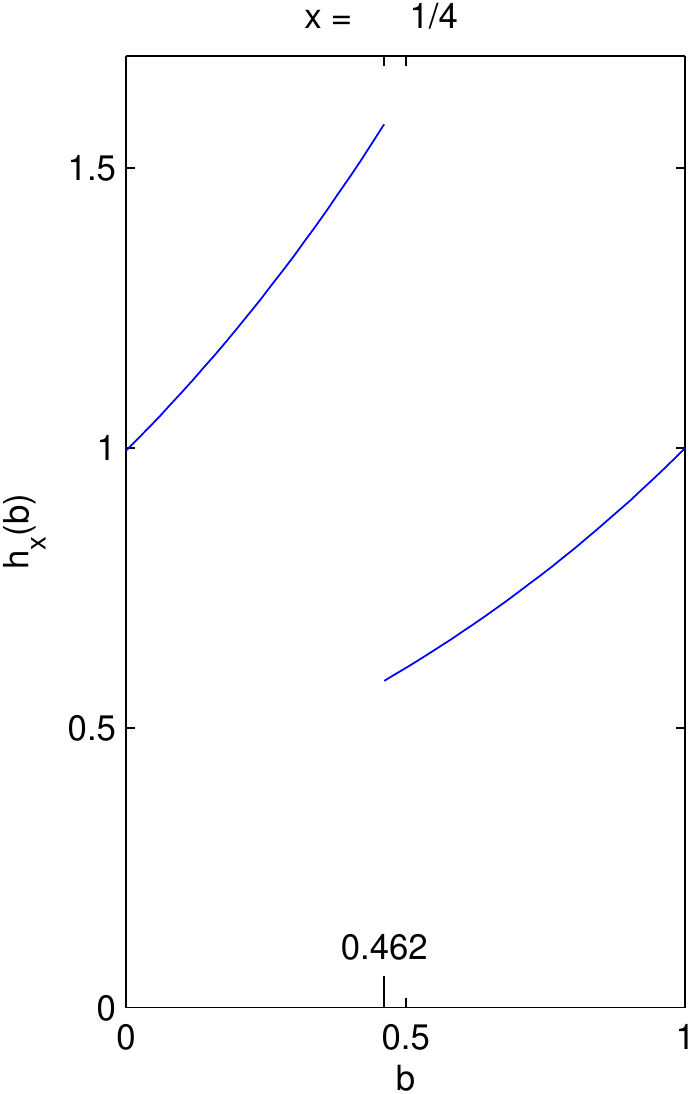}
\includegraphics[scale=.55]{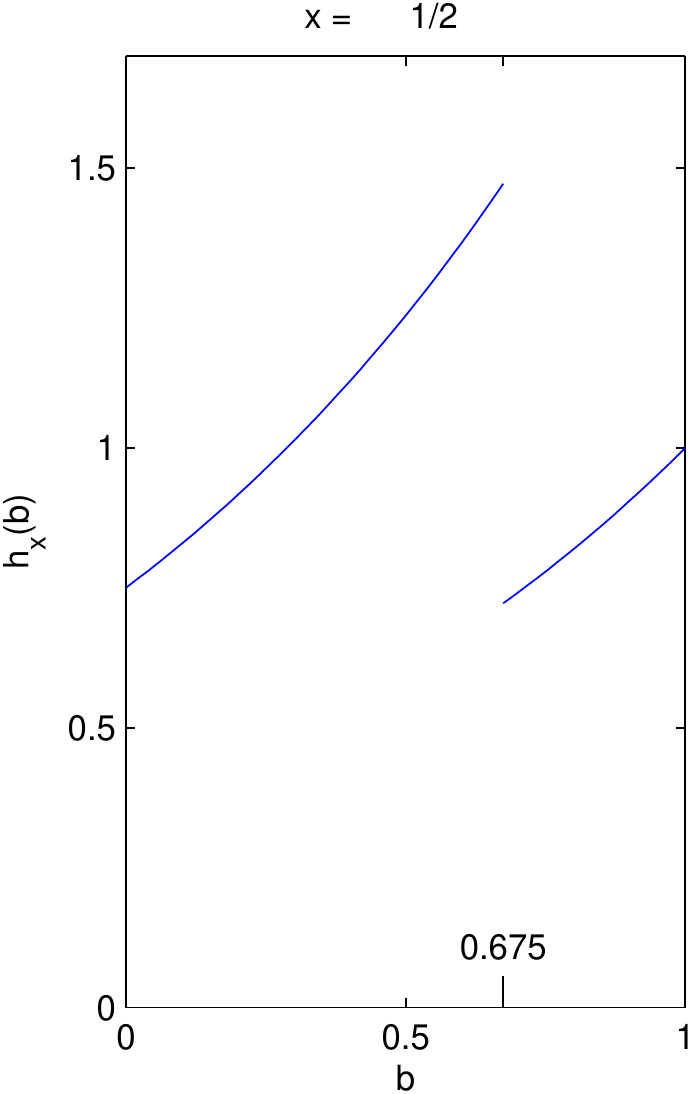}
\includegraphics[scale=.55]{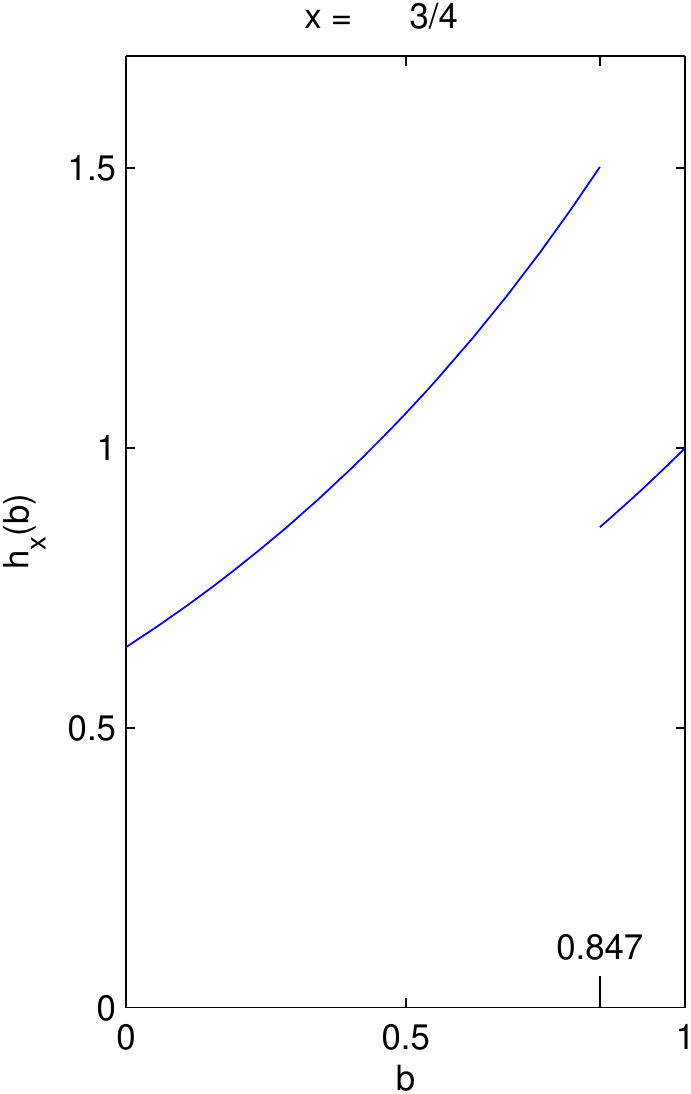}
\caption{Density for limiting rescaled   card number\newline occupying   a position around $xn$.}\label{Fi:h_x}
\end{figure}

\medskip
We now turn to the study of the entire distribution $p_n(\text{id},\cdot)$.
We need to introduce some additional concepts and notation. Fix a positive integer $n$. Let $l=(l_1,\cdots, l_{n-1})$
be an $(n-1)$-vector  of  positive integers
satisfying $i\le l_i\le n-1$. Consider the collection of all integer-valued paths
$\{Y_i\}_{i=1}^n$ satisfying $1\le Y_1\le Y_2\cdots\le Y_n=n$, with the strict inequality $Y_{i+1}>Y_i$ holding if
$Y_i\le l_i$. Note that  one always has  $Y_i\ge i$, for all $i=1,\cdots, n$.
 Call such paths \it nondecreasing $l$-paths of length $n$.\rm\ Denote the number of such paths by $N_n(l)$.
Note that $N_n(l)$ is strictly decreasing in each of its $n-1$ variables.

Recall that for $\sigma\in S_n$ and  $i,j\in[n]$ with $i<j$, the pair $(i,j)$ is called an \it\ inversion\rm\ for $\sigma$ if $\sigma_j<\sigma_i$.
According to our convention, $(i,j)$ is an inversion for $\sigma$ if the card in \it position $i$\rm\
has a higher number than the card in
\it position\rm\  $j$.
Thus, $(i,j)$ is an inversion for the inverse permutation $\sigma^{-1}$ if the card \it numbered\rm\ $i$ appears to the right
of the card \it numbered\rm\ $j$ in the permutation $\sigma$.
In this case, as we have already noted before Corollary \ref{pairs},  we also say that the cards with numbers $i$ and $j$ form an inversion for $\sigma$.
For $2\le j\le n-1$ and $\sigma\in S_n$, let
$$
\begin{aligned}
 & I_j(\sigma)=\sum_{k=1}^{j-1}1_{\sigma^{-1}_k>\sigma^{-1}_j}=
 \text{\# of inversions in}\ \sigma\\
&  \text{  involving the card numbered }\  j\ \text{ and a card  numbered less than}\ j.
\end{aligned}
$$

Define $l(\sigma)=(l_1(\sigma),\cdots,l_{n-1}(\sigma))$, where
 $l_{n-1}(\sigma)=n-1$ and
$$
l_j(\sigma)=j+I_{n-j}(\sigma),\  j=1,\cdots, n-2.
$$
Note that $j\le l_j(\sigma)\le n-1,\  j=1,\cdots, n-2$.

\begin{lemma}\label{counting}
For each
 $l=(l_1,\cdots, l_{n-1})$ satisfying $j\le l_j\le n-1$, for $j=1,\cdots, n-1$, there exist exactly $n$ permutations $\sigma\in S_n$
 satisfying $l(\sigma)=l$.
\end{lemma}
\begin{proof}
Note that $l(\sigma)$ does not depend on $\sigma^{-1}_n$, the position in $\sigma$ of the card numbered $n$.
It is easy to see that any $\sigma\in S_n$ is uniquely determined by
the value of  $\sigma^{-1}_n$ and by the condition $l(\sigma)=l$,
where $l=(l_1,\cdots,l_{n-1})$ is as in the statement of the lemma.
\end{proof}

\begin{theorem}\label{basic}
Let $\sigma\in S_n$.
One has
$$
p_n(\text{id},\sigma)=\frac{N_n(l(\sigma))}{n^n}
$$
where $N_n(l)$ denotes the number of nondecreasing $l$-paths of length $n$.
\end{theorem}
Theorem \ref{basic}  gives a  qualitative picture of the nature of the bias in the $p_n(\text{id},\cdot)$-shuffle.
Indeed, using the strict monotonicity of $N(l)$ and the definition of $l(\sigma)$,
the following corollary is immediate from Theorem \ref{basic}.
\begin{corollary}\label{inverting}
\noindent i. $p_n(\text{id},\sigma)$ does not depend on $\sigma^{-1}_n$, the position in $\sigma$ of  card
number $n$;

\noindent ii.
Let $\sigma',\sigma''\in S_n$.
If $I_j(\sigma')\le I_j(\sigma'')$,
for all $j\in\{2,\cdots, n-1\}$, then
$$
p_n(\text{id},\sigma')\ge p_n(\text{id},\sigma''),
$$
with a strict inequality holding if $I_j(\sigma')< I_j(\sigma'')$, for some $j\in\{2,\cdots, n-1\}$.
\end{corollary}

\noindent \bf Remark.\rm\
Of course, we don't need the theorem to get part (i) of the corollary.
From the definition of the shuffle, it is clear that the distribution of card number n
is uniform.
Part (ii) shows in particular that among cards numbered from 1 to $n-1$, if every such pair
 of cards that forms an inversion for $\sigma'$ also forms an inversion for $\sigma''$,
then $p_n(\text{id},\sigma')\ge p_n(\text{id},\sigma'')$.
Thus, in the above sense,  the more a permutation preserves the order defined by id, but ignoring card number $n$, the more it is favored
by
$p_n( \text{id},\cdot)$.
{\it We have qualified the above sentence with the words ``in the above sense,'' because
 Corollary \ref{pairs} shows that if $n$ is  large and   $b_2>b_1>\tilde b\approx .768$, with $b_2$ close to $b_1$,
 then  $p_n(\text{id},\cdot)$  assigns a  probability greater than $\frac12$  to
 those permutations for which card number $[b_1n]$ and card number $[b_2n]$ form an inversion!}

It  seems quite difficult to estimate $N_n(l)$ for general $l$. However, the maximum and minimum over $l$ can be calculated
explicitly.

\begin{theorem}\label{estimate}
One has
\begin{equation}\label{cat}
2^{n-1}\le
N_n(l)\le \frac1{n+1}\binom {2n}n,
\end{equation}
for all $l=(l_1,\cdots,l_{n-1})$, with $j\le l_j\le n-1$, for $j=1,\cdots, n-1$.
The left hand inequality above is an equality if and only if $l_j=n-1$, for all $j=1,\cdots, n-1$,
and the right hand inequality above is an equality if and only if $l_j=j$, for all $j=1,\cdots, n-1$.
\end{theorem}
\begin{remark}
Note that the right hand term in \eqref{cat} is equal to $C_n$, the $n$th Catalan number.
\end{remark}
The  following corollary is immediate  from  Theorems
\ref{basic} and \ref{estimate}.
\begin{corollary}\label{3}
One has
\begin{equation}\label{comb}
\frac{2^{n-1}}{n^n}\le p_n(\text{id},\sigma)\le\frac{\binom{2n}n}{(n+1)n^n}.
\end{equation}
The right hand inequality above is an equality if and only if $\sigma$ possesses the increasing
subsequence $\{1,\cdots, n-1\}$, and the left hand inequality is an equality if and only if
$\sigma$ possesses the decreasing subsequence $\{n-1,\cdots, 1\}$.
\end{corollary}

Note that
the left hand side of \eqref{comb} is $\frac12\frac{2^n}{n^n}$ and the right hand side of \eqref{comb} behaves asymptotically
as $n\to\infty$  like $\frac1{\sqrt\pi n^\frac32}\frac{4^n}{n^n}$, while
the uniform probability measure $U_n(\sigma)=\frac1{n!}$ behaves asymptotically as $n\to\infty$ like
 $\frac1{\sqrt{2\pi n}}\frac{e^n}{n^n}$.
Thus, we have the following tight uniform bounds over $\sigma\in S_n$:
$$
(1+o(1))(\frac{\pi n}2)^\frac12(\frac2e)^n \le \frac{p_n( \text{id},\sigma)}{U_n(\sigma)}\le(1+o(1))\frac{\sqrt2}n(\frac4e)^n,
\ \text{as}\ n\to\infty.
$$
In particular, the \it separation distance\rm\ between $U_n$ and $p_n(\text{id},\cdot)$ approaches 1 exponentially fast as $n\to\infty$.
(Recall that the separation distance $s$ is defined by $s(U_n,p_n(\text{id},\cdot))=\max_{\sigma\in S_n}(1-\frac{p_n(\text{id},\sigma)}{U_n(\sigma)})$.)


We now pose a question.

\medskip

\noindent \bf Question.\rm\
Consider the random walk with increment distribution given by $p_n(\cdot,\cdot)$.
Letting $(p_n)^{(m)}( \text{id},\cdot)$ denote the $m$-fold convolution of $p_n(\text{id},\cdot)$,  which is the
distribution of the random walk at time $m$ given that it started from id,
how large must   $\{m_n\}_{n=1}^\infty$ be so that
 $\lim_{n\to\infty}||U_n-(p_n)^{m_n}( \text{id},\cdot)||_{\text{TV}}$ equals 0,
and how small  must $\{m_n\}_{n=1}^\infty$  be so that it equals 1?

In light of the discussion below, one would expect that $m_n$ will be on the order $\log n$.
In order to use Theorem \ref{basic} to answer this question,  one needs good bounds  on  $N_n(l)$ for general $l$.
This seems to be  a quite difficult combinatorial problem.
It follows from Theorem \ref{basic} that this random walk is not reversible.

\begin{corollary}\label{2}
The random walk on $S_n$ with increment transition measure  $p_n(\text{id},\cdot)$ is not reversible.
\end{corollary}
\begin{proof}
From the formula in Theorem \ref{basic}, it is easy to see that the equality $p_n(\text{id},\sigma)=p_n(\text{id},\sigma^{-1})$
does not hold for all $\sigma\in S_n$.
\end{proof}

We prove
Theorems \ref{First}-\ref{estimate} in sections 2-7 respectively.
The proofs of Corollaries \ref{joint}, \ref{pairs} and \ref{densityatpos} are given immediately after the proof of Theorem \ref{general}.
\medskip

The original  motivation for this paper comes from the results on mixing times for a number of classical shuffles; in particular,
the \it random to random insertion\rm\ shuffle,  a random walk on $S_n$ whose transition is implemented by choosing
 a card at random, removing it from the row, and
then reinserting it in a random position in the row.
Denote this random walk by $\{X_m\}_{m=0}^\infty$ and
 let $P^{(n)}_\sigma$ denote probabilities for the random walk
starting from $\sigma$. The random walk is irreducible and the
uniform distribution $U_n$ is its invariant measure. It's
aperiodic since $P_\sigma^{(n)}(X_1=\sigma)=\frac1n$. Thus
$P_{\text{id}}^{(n)}(X_m\in\cdot)$ converges to $U_n$ as
$m\to\infty$. One is interested in the rate of   convergence
in the total variation norm as the parameter $n$ grows.
It is known that the mixing time is on the order $n\log n$. A long-standing open problem is to establish the \it cut-off phenomenon\rm; namely to establish
the  existence
of a $c^*$ such that if $m_n\ge cn\log n$ with $c>c^*$, then
$\lim_{n\to\infty}||P_{\text{id}}^{(n)}(X_{m_n}\in\cdot)-U_n||_{\text{TV}}=0$, and if
$m_n\le cn\log n$ with $c<c^*$, then
$\lim_{n\to\infty}||P_{\text{id}}^{(n)}(X_{m_n}\in\cdot)-U_n||_{\text{TV}}=1$.
It has been conjectured that $c^*=\frac34$, and the lower bound $c*\ge\frac34$ has been proven very recently using
delicate probabilistic estimates \cite{S}.
The best know upper bound is $c^*\le 2$, which was obtained by analytic methods \cite{SZ}.
For other similar looking
 shuffles, such as the \it random transposition
shuffle\rm\ (where at each stage, two cards are selected independently---so the same card might be selected twice---and
then their positions are swapped)  and the \it top to random insertion shuffle \rm\ (where at each stage, the current top
card (left-most card in our setup) is removed and randomly reinserted),
the cut-off phenomenon has been proven with $m_n$ in the same form
as above, with $c^*=\frac12$ and $c^*=1$ respectively \cite{D}.

Note that the mixing times of all the shuffles above are on the order $n\log n$.
Now recall that the coupon collector's problem is the problem of determining how many samples of an IID random variable, distributed
uniformly on $[n]$, are required until every number has been selected at least once. Denoting the required number of samples by $T_n$,
it is well known that  $\lim_{n\to\infty}P(T_n\ge n\log n+c_nn)$ equals 0 if $\lim_{n\to\infty}c_n=\infty$
and equals 1 if  $\lim_{n\to\infty}c_n=-\infty$. More delicate estimates show that if $T_{n;k}$ denotes the number of samples
required until all but $k$ cards are selected once, then
 $\lim_{n\to\infty}P(T_{n,n^l}\ge (1-l)n\log n+c_nn)$ equals 0 or 1 with $c_n$ as above.
The coupon collector phenomenology is an integral part of the proofs of some of the results noted above.
This leads one to wonder whether the order $n\log n$  for mixing in the above shuffles is caused exclusively by the coupon
collector's phenomenology, that is exclusively by the fact that
one needs order $n\log n$ shuffles to move most of the cards at least once,
 or whether this order is inherent in these shuffles
for additional reasons.
(Indeed, after order $n\log n$ shuffles, most of the cards have been removed and reinserted many times.)
It was natural then to consider a shuffle that moved every card exactly once.
To make such a model as close as possible in spirit to the random to random insertion shuffle, one should randomize  the order
in which the $n$ cards are removed and reinserted exactly once. However, this seemed intractable, so we were led to study the problem
presented in this paper, where the order in which the cards are removed and reinserted is not random, but rather is the original
left to right order of the cards. We admit that this is no longer the appropriate model, however, we think the results
 obtained here are of independent interest.
As was noted, the fact that $n!\nmid n^n$ when $n\ge3$ shows immediately that the distribution of our shuffle
 cannot be uniform after one shuffle.
If one randomizes the order in which the $n$ cards are removed and reinserted, then this argument breaks down.
However, even this shuffle does not give the uniform distribution; indeed,
one can check by hand
that for $n=3$, the resulting  probabilities  can take on the values
 $\frac{26}{162},\frac{27}{162}$ and $\frac{28}{162}$.

The reason we use the terminology \it card-cyclic \rm\ is that in the card-shuffling literature
the term  \it cyclic to random shuffle \rm\ (by which one means cyclic to random transposition shuffle) is used for the  shuffle
where  at step $k$ one takes the card  \it currently in position\rm\  $k\ \text{mod}\ n$ and transposes it with a random card. This kind
of shuffle is  \it position-cyclic,\rm\ whereas ours is \it card-cyclic.\rm\
In position cyclic shuffles, after one cycle, there are usually many cards that have not been moved at all.
For results on position-cyclic  to random transposition shuffles in the spirit of some of the results in this paper, see
\cite{RB}, \cite{SS}, \cite{GM}. For results on position-cyclic to random transposition shuffles in the spirit
of  the question we posed above, see  \cite{M} and \cite{MPS}.
\section{Proof of Theorem \ref{First}}
We first derive the exact combinatorial formula for $p_n(\text{id},\{\sigma_1=j\})$.
Of course we have $p_n(\text{id},\{\sigma_1=n\})=\frac1n$.
Now consider $1\le j\le n-1$.
If  card number $j$ is moved to the $k$-th position, with $2\le k\le n-j+1$, then   at the end of the shuffle it will be in the first position
if and only if the following occur. Cards numbered 1 up to $j-1$, which were moved before card number $j$ was moved,
must move successively to the right of card number $j+k-1$. If this occurs, then after card number $j$ is moved to position $k$,
the cards numbered $j+1$ up to $j+k-1$ will be to the left of card number $j$. These  cards numbered $j+1$ to $j+k-1$
 now must move successively to the right
of card number $j$. If this occurs, then card number $j$ will be in the first position. Now cards numbered
$j+k$ up to $n$ must all move to positions greater or equal to two, so that card number $j$ remains in the first position.
We now calculate the probability of this occurring.
The probability that cards numbered  1 up to $j-1$ move successively  to the right of card number $j+k-1$ is
$\prod_{l=2}^{j} \frac{n-j-k+l}n$. The probability that cards numbered $j+1$ to $j+k-1$, which occupy
the first $k-1$ positions,   move successively
to the right of card number $j$, which occupies the $k$-th position, is
$\prod_{l=1}^{k-1}\frac{n-k+l}n$. The probability that cards numbered
$j+k$ up to $n$ all move to positions greater or equal to two is $(\frac{n-1}n)^{n-j-k+1}$.
Thus, conditioned on card number $j$ moving to position $k$, with $2\le k\le n-j+1$, the probability
that card number $j$ will end up in the first position is
$\frac{(n-1)!}{(n-j-k+1)!}\frac{(n-1)^{n-j-k+1}}{n^{n-1}}$.
Conditioned on card number  $j$ moving  to position $k$ with $k>n-j+1$, the above considerations
show that the probability of it ending up in the first position is zero.

Now consider the case that $k=1$; that is,  $j$ is moved to the first position.
At the end of the shuffle,  card number $j$ will be in the first position if and only if
the following occur. Cards numbered 1 up to $j-1$ may move unrestrictedly.
Then after card number $j$ is moved to the first position, cards numbered
$j+1$ to $n$ must move to positions greater or equal to two, so that card number $j$ remains in the
first position. Thus, conditioned on  card $j$ moving to the first position, the probability that it will
end up in the first position is $(\frac{n-1}n)^{n-j}$.

From the above considerations and calculations, we conclude that
\begin{equation}\label{combform}
p_n(\text{id},\{\sigma_1=j\})=\frac1n(\frac{n-1}n)^{n-j}+
\frac{(n-1)!}{n^n}\sum_{k=2}^{n-j+1}\frac{(n-1)^{n-j-k+1}}{(n-j-k+1)!}.
\end{equation}

We now prove each of the three parts of the theorem.

\noindent \it Proof of (iii).\rm\
Consider first the case that $j=d_nn^\frac12$, with $\lim_{n\to\infty}d_n=d\in[0,\infty)$.
With a small change in notation, the proof also works with $d=\infty$.
We break up the sum in \eqref{combform} into three parts.
Fix a large $M$.
We look at the sum as $k$  runs from 2 to $[Mn^\frac12]$, from $[Mn^\frac12]+1$ to $[\frac12n]$, and from $[\frac12n]+1$ to $n-j+1$.
We begin with the last sum. Let $k=[(1-c)n]$ with $c\in(0,\frac12)$.
By looking at the ratio of two consecutive terms, it follows that
for $0\le x\le n-2$, the expression $\frac{(n-1)^x}{x!}$ is increasing in $x$.
Thus,
\begin{equation*}
\frac{(n-1)^{n-j-k+1}}{(n-j-k+1)!}\le\frac{(n-1)^{\frac n2}}{[\frac n2]!}\sim\frac{(2e)^\frac n2}{\sqrt{\pi en}},\ \text{as}\ n\to\infty.
\end{equation*}
Using this along with the fact that   $\frac{(n-1)!}{n^n}\sim\frac1ne^{-n}\sqrt{2\pi n}$ as $n\to\infty$,
we have
\begin{equation}\label{third}
\frac{(n-1)!}{n^n}\sum_{k=[\frac12n]+1}^{n-j+1}\frac{(n-1)^{n-j-k+1}}{(n-j-k+1)!}
\le\frac1{\sqrt{2e}}(\frac2e)^\frac n2,\ \text{for large}\ n.
\end{equation}

We now consider the first sum, as $k$ runs from 2 to $[Mn^\frac12]$.
For
$k=[cn^\frac12]$,  with $c\in(0,M]$, we write
\begin{equation}\label{calc1}
\begin{aligned}
&\frac{(n-1)^{n-j-k+1}}{(n-j-k+1)!}=\\
&\big(\frac{n-1}{n-d_nn^\frac12-cn^\frac12+1}\big)^{(n-d_nn^\frac12-cn^\frac12+1)}~
\frac{(n-d_nn^\frac12-cn^\frac12+1)^{(n-d_nn^\frac12-cn^\frac12+1)}}{(n-d_nn^\frac12-cn^\frac12+1)!}.
\end{aligned}
\end{equation}
As $n\to\infty$, we have
\begin{equation}\label{calc2}
\begin{aligned}
&(n-d_nn^\frac12-cn^\frac12+1)\log(\frac{n-1}{n-d_nn^\frac12-cn^\frac12+1})=\\
&(n-d_nn^\frac12-cn^\frac12+1)\log(1+\frac{d_nn^\frac12+cn^\frac12-2}{n-d_nn^\frac12-cn^\frac12+1})=\\
& (d_nn^\frac12+cn^\frac12-2)-\frac{(d_n+c)^2}2+O(n^{-\frac12}),
\end{aligned}
\end{equation}
where the term $O(n^{-\frac12})$ is uniform over $c\in(0,M]$.
Using \eqref{calc2} in \eqref{calc1}, we have as $n\to\infty$,
\begin{equation*}
\frac{(n-1)^{n-j-k+1}}{(n-j-k+1)!}\sim
e^{(d_n+c)n^\frac12-2-\frac{(d_n+c)^2}2}~\frac{e^{(n-d_nn^\frac12-cn^\frac12+1)}}{\sqrt{2\pi n}}, \ j=d_nn^\frac12, k=cn^\frac12,
\end{equation*}
and then
\begin{equation*}
\frac{(n-1)!}{n^n}\frac{(n-1)^{n-j-k+1}}{(n-j-k+1)!}\sim\frac1{ne}e^{-(d_n+c)^2}, \ j=d_nn^\frac12, k=cn^\frac12.
\end{equation*}
Thus, as $n\to\infty$,
\begin{equation}\label{fornext}
\frac{(n-1)!}{n^n}\sum_{k=2}^{[Mn^\frac12]}\frac{(n-1)^{n-j-k+1}}{(n-j-k+1)!}\sim\frac1{ne}\sum_{k=2}^{[Mn^\frac12]}
e^{-\frac12(d+\frac k{\sqrt n})^2},
\end{equation}
from which it follows that
\begin{equation}\label{1st}
\lim_{n\to\infty}n^\frac12\frac{(n-1)!}{n^n}\sum_{k=2}^{[Mn^\frac12]}\frac{(n-1)^{n-j-k+1}}{(n-j-k+1)!}=\frac1e\int_0^Me^{-\frac12(d+y)^2}dy.
\end{equation}

We now consider the second sum, as $k$ runs from $[Mn^\frac12]+1$ to $[\frac12n]$.
For $x\in(0,2]$, one has $\log (1+x)\le x-\frac1{18}x^2$. Using this and the fact that
 $0<\frac{j+k-2}{n-j-k+1}<2$, for large $n$, we have for large $n$,
\begin{equation}\label{second}
\begin{aligned}
&\log(\frac{n-1}{n-j-k+1})=\log(1+\frac{j+k-2}{n-j-k+1})\le\\
&\frac{j+k-2}{n-j-k+1}-\frac1{18}(\frac{j+k-2}{n-j-k+1})^2, \ [Mn^\frac12]+1\le k\le [\frac12n].
\end{aligned}
\end{equation}
Using \eqref{second}, we have as $n\to\infty$,
\begin{equation}\label{second2}
\begin{aligned}
&\frac{(n-1)^{n-j-k+1}}{(n-j-k+1)!}=(\frac{n-1}{n-j-k+1})^{n-j-k+1}\frac{(n-j-k+1)^{n-j-k+1}}{(n-j-k+1)!}\le\\
&\frac{(n-j-k+1)^{n-j-k+1}}{(n-j-k+1)!}e^{\big(j+k-2-\frac1{18}\frac{(j+k-2)^2}{n-j-k+1}\big)}\le\\
&(1+o(1))\frac{e^{n-j-k+1}}{\sqrt{2\pi(n-j-k+1)}}e^{\big(j+k-2-\frac{(j+k-2)^2}{18n}\big)}, \ \text{for}\ [Mn^\frac12]+1\le k\le [\frac12n].
\end{aligned}
\end{equation}
From \eqref{second2}, we obtain
\begin{equation}\label{second3}
\begin{aligned}
&\frac{(n-1)!}{n^n}\frac{(n-1)^{n-j-k+1}}{(n-j-k+1)!}\le\frac1n\big(1+o(1)\big)\frac{\sqrt2}ee^{-\frac{(j+k-2)^2}{18n}},\\
& \text{for}\ [Mn^\frac12]+1\le k\le[\frac12n],
\ \text{as}\ n\to\infty.
\end{aligned}
\end{equation}
Thus, similar to \eqref{fornext} and \eqref{1st}, we conclude that
\begin{equation}\label{finalsecond}
\limsup_{n\to\infty}n^\frac12\frac{(n-1)!}{n^n}\sum_{k=[Mn^\frac12]+1}^{[\frac12n]}\frac{(n-1)^{n-j-k+1}}{(n-j-k+1)!}\le
\frac{\sqrt2}e\int_M^\infty e^{-\frac1{18}(d+y)^2}dy.
\end{equation}
Using  \eqref{third}, \eqref{1st}, \eqref{finalsecond} and \eqref{combform}, and letting $M\to\infty$,  we
conclude that
$$
\lim_{n\to\infty}n^\frac12p_n(\text{id},\{\sigma^{-1}(1)=d_nn^\frac12\})=\frac1e\int_d^\infty e^{-\frac12y^2}dy.
$$

To prove the final statement in part (iii), we need to show that
\begin{equation}\label{integral}
\int_0^\infty(\int_x^\infty e^{-\frac12 y^2}dy)dx=1.
\end{equation}
Note that $F(x)=\sqrt\frac2\pi\int_0^xe^{-\frac12y^2}dy$ is the distribution of $|Z|$, where $Z\sim\text{N}(0,1)$.
Thus
\begin{equation}\label{distfunc}
\int_0^\infty(1-F(x))dx=E|Z|=\sqrt\frac2\pi\int_0^\infty x\sqrt\frac2\pi e^{-\frac12x^2}dx=\sqrt\frac2\pi.
\end{equation}
But $1-F(x)=\sqrt\frac2\pi\int_x^\infty e^{-\frac12y^2}dy$. Substituting this in \eqref{distfunc} gives
\eqref{integral}.
\medskip

\noindent \it Proof of part (i).\rm\
Now consider the case that $j=b_nn$ with $\lim_{n\to\infty}b_n=b\in(0,1]$.
As noted above, $\frac{(n-1)^x}{x!}$ is increasing for $0\le x\le n-2$. Thus, letting $\delta=\frac b2>0$, for sufficiently large $n$, one has for all $k$,
\begin{equation*}
\frac{(n-1)^{(n-j-k+1)}}{(n-j-k+1)!}\le\frac{(n-1)^{(n-1)(1-\delta)}}{[(n-1)(1-\delta)]!}\sim\frac1{\sqrt{2\pi n(1-\delta)}}(\frac e{1-\delta})^{(1-\delta)(n-1)},
\end{equation*}
and then for some constant $C_0>0$,
\begin{equation*}
\frac{(n-1)!}{n^n}\frac{(n-1)^{(n-j-k+1)}}{(n-j-k+1)!}\le \frac{C_0}n\big((1-\delta)^{(1-\delta)}e^\delta\big)^{-n}.
\end{equation*}
One can check that $e^x(1-x)^{1-x}$, with $x\in[0,1)$ attains its minimum  value at $x=0$, where it equals 1.
Thus,  for some $\epsilon>0$,
we have
\begin{equation}\label{ordern}
\frac{(n-1)!}{n^n}\sum_{k=2}^{n-j+1}\frac{(n-1)^{n-j-k+1}}{(n-j-k+1)!}\le\frac{C_0}{(1+\epsilon)^n}.
\end{equation}
Using \eqref{ordern} along with \eqref{combform}, it follows that
$$
\lim_{n\to\infty}np_n(\text{id},\{\sigma_1=b_nn\})=e^{b-1}.
$$
\medskip

\noindent \it Proof of part (ii).\rm
We now consider the case that $j=b_nn$ with $\lim_{n\to\infty}b_n=0$ and with $\liminf_{n\to\infty}\frac{n^{\frac12}}{\sqrt{\log n}}b_n>\sqrt2$.
For some $\epsilon>0$ and large $n$, we can write $j=l_nn^\frac12$, with $l_n\ge \sqrt{2(1+\epsilon)\log n}$ and $l_n=o(n^\frac12)$.
Since $\frac{(n-1)^x}{x!}$ is increasing for $0\le x\le n-2$, we have for
sufficiently large  $n$ and all $k$ that
\begin{equation}\label{intermediate}
\frac{(n-1)^{(n-j-k+1)}}{(n-j-k+1)!}\le\frac{(n-1)^{n-l_nn^\frac12}}{(n-l_nn^\frac12)!}
\sim(1+\frac{l_nn^\frac12-1}{n-l_nn^\frac12})^{(n-l_nn^\frac12)}\frac{e^{n-l_nn^\frac12}}{\sqrt{2\pi n}}.
\end{equation}
We have
\begin{equation}\label{logagain}
(n-l_nn^\frac12)\log(1+\frac{l_nn^\frac12-1}{n-l_nn^\frac12})=l_nn^\frac12-1-\frac12l_n^2(1+o(1)), \ \text{as}\ n\to\infty.
\end{equation}
From \eqref{intermediate} and \eqref{logagain}, it follows that for some $C_1>0$,
\begin{equation*}
\frac{(n-1)!}{n^n}\frac{(n-1)^{(n-j-k+1)}}{(n-j-k+1)!}\le \frac{C_1}ne^{-\frac12l_n^2}\le\frac{C_1}{n^{2+\epsilon}},
\end{equation*}
and thus
\begin{equation}\label{finalinter}
\frac{(n-1)!}{n^n}\sum_{k=2}^{n-j+1}\frac{(n-1)^{(n-j-k+1)}}{(n-j-k+1)!}\le \frac{C_1}{n^{1+\epsilon}}.
\end{equation}
From  \eqref{finalinter} with \eqref{combform}, it follows that
$$
\lim_{n\to\infty}np_n(\text{id},\{\sigma_1=b_n\})=e^{-1}.
$$
\hfill $\square$

\section{Proof of Theorem \ref{Last}}
We first derive the exact combinatorial formula for $p_n(\text{id},\{\sigma_n=j\})$.
Of course, $p_n(\text{id},\{\sigma_n=n\})=\frac1n$. Now consider $1\le j\le n-1$.
If card number $j$ is moved to the $k$-th position, with $j\le k\le n$,
then at the end of the shuffle it will be in the last position if and only if the following occur.
Cards numbered 1 to $j-1$, which were moved before card number $j$ was moved, must all move to the left of card number
$k+1$ (if $k=n$,  these cards can move unrestrictedly). If this occurs, then after card number $j$ is moved to position $k$, the cards numbered $1,\cdots, j-1$
and $j+1,\cdots, k$  will be to the left
of card number $j$. Now cards numbered $j+1,\cdots, k$ must all move to positions smaller or equal to $k-1$ in order that they remain
to the left of card number $j$. And then cards numbered $k+1,\cdots, n$ must successively move to the left of card number
$j$ (if $k=n$, this step is vacuous).  We now calculate the probability of this occurring. The probability that cards numbered
1 up to $j-1$ move to left of card number $k+1$ is $(\frac kn)^{j-1}$.
The probability that cards numbered $j+1,\cdots, k$, which are all in positions smaller than or equal to $k-1$, will all move
to positions  smaller than or equal to $k-1$ is $(\frac{k-1}n)^{k-j}$. The probability that cards numbered
$k+1,\cdots, n$, which occupy the positions $k+1,\cdots, n$, move successively to the left of card number $j$ which occupies the $k$-th position,
is $\prod_{l=k}^{n-1}\frac ln$. Thus, conditioned on card number $j$ moving to position $k$, with $j\le k\le n$, the probability
that card number $j$ will end up in the last position is
$\frac{k^{j-1}(k-1)^{k-j}}{(k-1)!}\frac{(n-1)!}{n^{n-1}}$. Conditioned on card number $j$ moving  to position $k$ with $1\le k\le j-1$,
the above considerations show that  the probability of it ending up in the last position is zero.

From the above considerations and calculations, we conclude that
\begin{equation}
p_n(\text{id},\{\sigma_n=j\})=\frac{(n-1)!}{n^n}\sum_{k=j}^n\frac{k^{j-1}(k-1)^{k-j}}{(k-1)!},
\end{equation}
which we rewrite in the form
\begin{equation}\label{comboformlast}
p_n(\text{id},\{\sigma_n=j\})=\frac{(n-1)!}{n^n}\sum_{m=j-1}^{n-1}\big(\frac{m+1}m\big)^{j-1}~\frac{m^m}{m!},
\end{equation}
where $0^0$ and $(\frac10)^0$ are understood to be 1.
Note that the formula is also correct for $j=n$.

We prove the following estimate.
\begin{lemma}\label{mmm}
\begin{equation}
\sum_{m=1}^{n-1}\frac{m^m}{m!}\sim\frac1{(e-1)\sqrt{2\pi}}\frac{e^n}{\sqrt n}, \ \text{as}\ n\to\infty.
\end{equation}
\end{lemma}
\begin{proof}
We have
\begin{equation}\label{stir}
\sum_{m=1}^{n-1}\frac{m^m}{m!}\sim\frac1{\sqrt{2\pi}}\sum_{m=1}^{n-1}\frac{e^m}{\sqrt m},
\ \text{as}\  n\to\infty.
 \end{equation}
Let $a(m)=e^m$, $b(m)=m^{-\frac12}$ and $A(m)=\sum_{j=1}^ma(j)=\frac{e^{m+1}-e}{e-1}$, for $m\ge1$.
Recall the following summation by parts formula:
$$
\sum_{m=1}^{n-1}a(m)b(m)=\sum_{m=1}^{n-2}A(m)\big(b(m)-b(m+1)\big)+A(n-1)b(n-1).
$$
Thus we have
\begin{equation}\label{sumbyparts}
\sum_{m=1}^{n-1}\frac{e^m}{\sqrt m}=\frac{e^n-e}{(e-1)\sqrt{n-1}}+\sum_{m=1}^{n-2}\frac{e^{m+1}-e}{e-1}\big(\frac1{\sqrt{m}}-\frac1{\sqrt{m+1}}\big).
\end{equation}
Using  the mean value theorem, it follows that there exists a $C>0$, and for any $M\ge1$,   a $C_M>0$  such that
\begin{equation}\label{neglig}
\sum_{m=1}^{n-2}\frac{e^{m+1}-e}{e-1}\big(\frac1{\sqrt m}-\frac1{\sqrt{m+1}}\big)\le
C_M+\frac CM\sum_{m=M}^{n-2}\frac{e^m}{\sqrt m}
\end{equation}
From \eqref{sumbyparts} and \eqref{neglig} it follows that
\begin{equation}\label{almost}
\lim_{n\to\infty}\sqrt ne^{-n}\sum_{m=1}^{n-1}\frac{e^m}{\sqrt m}=\frac1{e-1}.
\end{equation}
Now \eqref{stir} and \eqref{almost} give
\begin{equation}\label{almost2}
\lim_{n\to\infty}\sqrt ne^{-n}\sum_{m=1}^{n-1}\frac{m^m}{m!}=\frac1{(e-1)\sqrt{2\pi}}.
\end{equation}
\end{proof}

We now consider successively each of the three parts of the theorem.

\it\noindent Proof of (i).\rm\
By Lemma \ref{mmm}, it follows that for any $\epsilon\in(0,1)$,
\begin{equation}\label{neg-eps}
\sum_{m=1}^{[(1-\epsilon)n]-2}\frac{m^m}{m!}=o(\sum_{m=1}^{n-1}\frac{m^m}{m!}),\ \text{as}\ n\to\infty;
\end{equation}
and thus, from Lemma \ref{mmm} again we have
\begin{equation}\label{rest-eps}
\sum_{m=[(1-\epsilon)n]-1}^{n-1}\frac{m^m}{m!}\sim\frac1{(e-1)\sqrt{2\pi}}\frac{e^n}{\sqrt n}, \ \text{as}\ n\to\infty.
\end{equation}
Also, $\max_{m\in\{j-1,\cdots, n-1\}}(\frac{m+1}m)^{j-1}=(\frac j{j-1})^{j-1}$, and  standard analysis shows that
$\sup_{j\ge2}(\frac j{j-1})^{j-1}=e$.
Thus, from \eqref{neg-eps} it follows that
\begin{equation}\label{estlast}
\sum_{m=j}^{n-1}(\frac{m+1}m)^{j-1}\frac{m^m}{m!}\sim\sum_{m=[(1-\epsilon)n]-1}^{n-1}(\frac{m+1}m)^{j-1}\frac{m^m}{m!} , \ \text{as}\ n\to\infty.
\end{equation}
Let  $j=b_nn$ with $\lim_{n\to\infty}b_n=b\in[0,1)$. Now substitute for $j$ in \eqref{estlast}.
Clearly,
$$
\begin{aligned}
&\lim_{\epsilon\to0}\liminf_{n\to\infty}\min_{m\in\{[(1-\epsilon)n]-1,\cdots, n-1\}}\big(\frac{m+1}m\big)^{b_nn-1}=\\
&\lim_{\epsilon\to0}\limsup_{n\to\infty}\max_{m\in\{[(1-\epsilon)n]-1,\cdots, n-1\}}\big(\frac{m+1}m\big)^{b_nn-1}=e^b.
\end{aligned}
$$
Using this along with \eqref{rest-eps} and \eqref{estlast}, we conclude that
\begin{equation}\label{keylasti}
\sum_{m=b_nn-1}^{n-1}(\frac{m+1}m)^{b_nn-1}\frac{m^m}{m!}\sim\frac{e^b}{(e-1)\sqrt{2\pi}}\frac{e^n}{\sqrt n}, \ \text{as}\ n\to\infty.
\end{equation}
Now \eqref{keylasti} and \eqref{comboformlast}
give
$$
p_n(\text{id},\{\sigma_n=b_nn\})\sim\frac1ne^{-n}\sqrt{2\pi n}\frac{e^b}{(e-1)\sqrt{2\pi}}\frac{e^n}{\sqrt n}=\frac1n\frac{e^b}{e-1},
\ \text{as}\ n\to\infty,
$$
which proves (i).
\medskip

\it\noindent Proof of (ii).\rm\ Let $j=b_nn$ with $\lim_{n\to\infty}b_n=1$ and $\lim_{n\to\infty}(n-b_nn)=\infty$.
We can rewrite $j$ in the form $j=n-\gamma_n$, where $\lim_{n\to\infty}\gamma_n=\infty$ and $\lim_{n\to\infty}\frac{\gamma_n}n=0$.
From \eqref{comboformlast} we have
\begin{equation}\label{ii}
p_n(\text{id},\{\sigma_n=b_nn\})=\frac{(n-1)!}{n^n}\sum_{m=n-\gamma_n-1}^{n-1}\big(\frac{m+1}m\big)^{n-\gamma_n-1}~\frac{m^m}{m!},
\end{equation}
For $m$ in the range appearing on the right hand side above, we have
$$
(\frac n{n-1})^{n-\gamma_n-1}\le (\frac{m+1}m)^{n-\gamma_n-1}\le (\frac{n-\gamma_n}{n-\gamma_n-1})^{n-\gamma_n-1},
$$
and both the left and the right hand sides above converge to  $e$ when $n\to\infty$.
Thus, from \eqref{ii}, we have
\begin{equation}\label{ii'}
p_n,\{\sigma_n=b_nn\})\sim \frac1ne^{-n}\sqrt{2\pi n}~e\sum_{m=n-\gamma_n-1}^{n-1}\frac{m^m}{m!}, \ \text{as}\ n\to\infty.
\end{equation}
By Lemma \ref{mmm}, $\sum_{m=1}^{n-1}\frac{m^m}{m!}\sim\frac1{(e-1)\sqrt{2\pi}}\frac{e^n}{\sqrt n}$ and
$\sum_{m=1}^{n-\gamma_n-1}\frac{m^m}{m!}\sim\frac1{(e-1)\sqrt{2\pi}}\frac{e^{n-\gamma_n}}{\sqrt{n-\gamma_n}}$,
as $n\to\infty$. By the assumption on  $\gamma_n$, the order of the latter term is smaller than
that of the former term; hence from Lemma \ref{mmm} again,
\begin{equation}\label{finalii}
\sum_{m=n-\gamma_n-1}^{n-1}\frac{m^m}{m!}\sim\frac1{(e-1)\sqrt{2\pi}}\frac{e^n}{\sqrt n},\ \text{as}\ n\to\infty.
\end{equation}
From \eqref{ii'} and \eqref{finalii}, it follows that
$$
p_n(\text{id},\{\sigma_n=b_nn\})\sim\frac1ne^{-n}\sqrt{2\pi n}~e~\frac1{(e-1)\sqrt{2\pi}}\frac{e^n}{\sqrt n}=\frac1n\frac e{e-1},
\ \text{as}\ n\to\infty,
$$
which proves (ii).
\medskip

\it\noindent Proof of (iii).\rm\ Now we let $j=n-l$ with $l\ge0$ fixed.
By \eqref{comboformlast}, we have
\begin{equation}
p_n(\text{id},\{\sigma_n=n-l\})=\frac{(n-1)!}{n^n}\sum_{m=n-l-1}^{n-1}\big(\frac{m+1}m\big)^{n-l-1}~\frac{m^m}{m!}.
\end{equation}
From this, it follows that
\begin{equation}\label{iii}
p_n(\text{id},\{\sigma_n=n-l\})\sim \frac1ne^{-n}\sqrt{2\pi n}~e\sum_{r=0}^l\frac{(n-r-1)^{n-r-1}}{(n-r-1)!}.
\end{equation}
We have $\frac{(n-r-1)^{n-r-1}}{(n-r-1)!}\sim\frac{e^{n-r-1}}{\sqrt{2\pi n}}$ as $n\to\infty$.
Thus, from \eqref{iii}, we conclude that
\begin{equation}
p_n(\text{id},\{\sigma_n=n-l\})\sim \frac1n\sum_{r=0}^le^{-r}=\frac1n\frac{e-e^{-l}}{e-1},
\end{equation}
which proves (iii).
\hfill $\square$

\section{Proof of Theorem \ref{TV}}
Let  $L,M>0$, with $L$ being an integer. In the calculations that follow, we will use the generic $P$ to denote probabilities of events
concerning the shuffling mechanism.
Let $B^{(n)}_M$
be the event that at least one out of the first $[Mn^\frac12]$ cards (that is, the cards numbered from 1 to  $[Mn^\frac12]$)
gets removed and reinserted in a position that is no greater than
$[Mn^\frac12]$.
Note that $P(B^{(n)}_M)=1-(1-\frac{[Mn^\frac12]}n)^{[Mn^\frac12]}$; so
\begin{equation}\label{TVfirst}
\lim_{n\to\infty}P(B^{(n)}_M)=1-e^{-M^2}.
\end{equation}
If the event $B^{(n)}_M$ occurs, let $j^{(n)}_M\le [Mn^\frac12]$ denote the number of the  card with the smallest number that gets
removed and reinserted in a position no greater than $[Mn^\frac12]$. For convenience, we define $j^{(n)}_M=\infty$
if the event $B^{(n)}_M$ does not occur; thus, $B^{(n)}_M=\{j^{(n)}_M\le [Mn^\frac12]\}$.

For $L<[Mn^\frac12]-1$, define $C^{(n)}_{M,L}$ to be  the event that no more than $L$ out of the first $[Mn^\frac12]$ cards
 are removed and reinserted in  a position
to the left of card number $2[Mn^\frac12]-L-1$ (by the restriction on $L$, card
number $2[Mn^\frac12]-[L]-1$  is guaranteed not to be among the first  $[Mn^\frac12]$ cards).
From the definitions, it is easy to see that
$$
P(C^{(n)}_{M,L})\ge P(X_{n,M,L}\le L),
$$
where $X_{n,M,L}\sim \text{Bin}([Mn^\frac12],\frac{2[Mn^\frac12]-L-2}n)$.
Since $EX_{n,M,L},\text{Var}(X_{n,M,L})\sim 2M^2$ as $n\to\infty$, it follows that
\begin{equation}\label{TVsecond}
\lim_{M\to\infty}\lim_{n\to\infty}P(C^{(n)}_{M,L(M)})=1,\ \text{if}\ L(M)\ge 3M^2.
\end{equation}

We claim that if $C^{(n)}_{M,L}$ occurs and  $j^{(n)}_M\le [Mn^\frac12]$,
then immediately after  card number $j^{(n)}_M$ is removed and reinserted,  there will be no more than
$L$ cards with numbers less than $j^{(n)}_M$ appearing to the left of card number $j^{(n)}_M$.
Indeed, assume to the contrary that at least $L+1$ cards with numbers
less than $j^{(n)}_M$ appear to the left of newly reinserted card number $j^{(n)}_M$.
But then since  $C^{(n)}_{M,L}$ has occurred, card
number $2[Mn^\frac12]-L-1$ is also necessarily to the left of newly reinserted card number
$j^{(n)}_M$. Since every  card with a number greater than
$[Mn^\frac12]$  has not yet been removed and reinserted, it follow that all these cards maintain their original relative order;
thus in fact all the cards from $[Mn^\frac12]+1$ up to
$2[Mn^\frac12]-L-1$ are to the left of newly reinserted card number $j^{(n)}_M$. We conclude that these
$[Mn^\frac12]-L-1$ cards as well as at least $L+1$ other cards
are to the left of newly inserted card number $j^{(n)}_M$; but this contradicts
the assumption that the position of card number $j^{(n)}_M$ is no greater than
$[Mn^\frac12]$.

If $j^{(n)}_M\le [Mn^\frac12]$, let $\text{pos}(j^{(n)}_M)$ denote its position immediately
after it is removed and reinserted. For the rest of this paragraph,
when we use the word ``now,'' we mean at the time immediately after card
$j^{(n)}_M$ is removed and reinserted. If $j^{(n)}_M\le[Mn^\frac12]$  and $C^{(n)}_{M,L}$ has occurred, then immediately after card number
 $j^{(n)}_M$ is removed and reinserted, it will find itself in position
 $\text{pos}(j^{(n)}_M)\le [Mn^\frac12]$, and the number of cards with lower numbers than $j^{(n)}_M$ that will be occupying
 positions to the left of  position $\text{pos}(j^{(n)}_M)$
will be between 0 and $L$; call this number $L'$.  All the cards with numbers higher than  $j^{(n)}_M$ will be in their original relative order;
thus,  $\text{pos}(j^{(n)}_M)-1-L'$ of them will be in positions to the left of  $\text{pos}(j^{(n)}_M)$, and
$n-j^{(n)}_M-\text{pos}(j^{(n)}_M)+1+L'$
of them will be in positions to the right of $\text{pos}(j^{(n)}_M)$.
Let $D^{(n)}_{M,L}$ denote the event that no more than $L$ out of these
 $\text{pos}(j^{(n)}_M)-1-L'$  cards that are now to the left of card $j^{(n)}_M$ in position  $\text{pos}(j^{(n)}_M)$ end up to the left of card $j^{(n)}_M$ after being removed and reinserted, and let
 $E^{(n)}_{M,L}$ denote the event that no more than $\rho L $ out of these
$n-j^{(n)}_M-\text{pos}(j^{(n)}_M)+1+L'$ cards that are now  to the right of card  $j^{(n)}_M$
in position $\text{pos}(j^{(n)}_M)$
 end up to the left of card $j^{(n)}_M$ after they are  finally removed and reinserted,  thereby ending the shuffle.
Here $\rho L$ is an integer which will be chosen later.

By looking at the worst case scenario (by choosing $L'=0$ and $\text{pos}(j^{(n)}_M)=[Mn^\frac12]$), it follows easily that
\begin{equation*}
P(D^{(n)}_{M,L}|C^{(n)}_{M,L}, B^{(n)}_M)\ge P(Y_{n,M,L}\le L),
\end{equation*}
where $Y_{n,M,L}\sim\text{Bin}([Mn^\frac12]-1,\frac{[Mn^\frac12]-1}n)$.
Since $EY_{n,M,L},\text{Var}(Y_{n,M,L})\sim M^2$ as $n\to\infty$,
it follows that
\begin{equation}\label{TVthird}
\lim_{M\to\infty}\lim_{n\to\infty}P(D^{(n)}_{M,L(M)}|C^{(n)}_{M,L(M)}, B^{(n)}_M)=1,\ \text{if}\ L(M)\ge 2M^2.
\end{equation}

Now we consider  $P(E^{(n)}_{M,L}|D^{(n)}_{M,L},C^{(n)}_{M,L}, B^{(n)}_M)$.
Conditioned on $B^{(n)}_{M},C^{(n)}_{M,L}$ and  $D^{(n)}_{M,L}$, when event
$D^{(n)}_{M,L}$ ends and event $E^{(n)}_{M,L}$ starts, the card
$j_M^{(n)}$ will be in a position between 1 and $2L+1$; call the position $k$.
Then the worst case scenario would be
to set $n-j^{(n)}_M-\text{pos}(j^{(n)}_M)+1+L'$  equal to $n-k$; that is, equal to the total number of cards
to the right of card $j^{(n)}_M$.
Thus a lower bound for $P(E^{(n)}_{M,L}|D^{(n)}_{M,L},C^{(n)}_{M,L}, B^{(n)}_M)$ is the
minimum over those $k$  between 1 and $2L+1$ of the
probability
that in a deck of $n$ cards ordered from 1 to $n$, if  one removes and randomly reinserts the last $n-k$ cards, then no more
than $\rho L$ of them get reinserted to the left of card $k$.
We can write these probabilities  in terms of  certain probabilities for  certain
geometric random variables. For any $i\ge 1$, let $T^{i}_{q_i}$ denote a geometric random variable with parameter $q_i$ and with values in $\{1,2,\cdots\}$,
and let $T^{i}_{q_i}$ and
$T^{j}_{q_j}$ be independent for $j\neq i$.
For a fixed $k$, the above probability is
$P(\sum_{l=0}^{\rho L}T_{\frac{k+l}n}>n-k)$.
To see this, think of the number of cards that are  removed and  randomly reinserted until the first time one  of them gets placed
to the left of card number $k$ as a $T^1_{\frac kn}$ random variable, think of
the number of  cards after the first one gets placed to the left of card number $j$ until a second one gets placed
to the left of card number $j$ as a $T^2_{\frac{k+1}n}$ random variable, etc.
(In fact, these numbers  are not distributed according to these random variables, because there are only
a finite number of cards. What is true precisely, for example, with regard to the first time a card gets
placed to the left of card number $k$ is that for $l\le n-k$,
the probability of needing exactly $l$ cards to be removed and reinserted until the first time one of them gets placed to the left of card number $k$
is equal to the probability that $T^1_{\frac kn}$ is equal to $l$.)

So we have
\begin{equation}\label{estE}
P(E^{(n)}_{M,L}|D^{(n)}_{M,L},C^{(n)}_{M,L}, B^{(n)}_M)\ge \min_{1\le k\le 2L+1}
P(\sum_{l=0}^{\rho L}T_{\frac{k+l}n}>n-k).
\end{equation}
Now for all $0\le k\le 2L+1$,
$$
E\sum_{l=0}^{\rho L}T_{\frac{k+l}n}=n\sum_{l=0}^{\rho L}\frac1{k+l}\ge n\log \frac{k+\rho L+1}k\ge
 n\log \frac{2L+2+\rho L}{2L+1},
$$
and
$$
\text{Var}(\sum_{l=0}^{\rho L}T_{\frac{k+l}n})\le  Cn^2,
$$
 for a constant
$C$ independent of $k$ and $L$.
Thus, by Chebyshev's inequality, for any $\lambda(L)$,
\begin{equation}\label{Cheb}
P(\sum_{l=0}^{\rho L}T_{\frac{k+l}n}\ge  n\log \frac{2L+2+\rho L}{2L+1}-n\lambda(L))\ge 1-\frac C{(\lambda(L))^2}.
\end{equation}
Choosing now  $\rho$ in the definition of $E^{(n)}_{M,L}$ sufficiently large so that
 $\log \frac{2L+2+\rho L}{2L+1}>2$,
 and letting $\lambda(L)=
\frac12\log \frac{2L+2+\rho L}{2L+1}$, it follows from \eqref{estE} and \eqref{Cheb} that
\begin{equation}\label{TVfourth}
\lim_{L\to\infty}\lim_{n\to\infty}P(E^{(n)}_{M,L}|D^{(n)}_{M,L},C^{(n)}_{M,L}, B^{(n)}_M)=1.
\end{equation}
If events  $B^{(n)}_{M},C^{(n)}_{M,L}$, $D^{(n)}_{M,L}$ and $E^{(n)}_{M,L}$ occur, then at the end of the shuffle,
card number $j^{(n)}_M\le [Mn^\frac12]$ will end up in a position between 1 and $2L+\rho L+1$.
Thus,  by \eqref{TVfirst}-\eqref{TVthird} and \eqref{TVfourth}, we conclude that
\eqref{TVcond} holds.

Finally, we have $U_n(A^{(n)}_{M,L})=1-\frac{\binom{n-[Mn^\frac12]}{L}}{\binom nL}$, from which it follows
that $\lim_{n\to\infty}U_n(A^{(n)}_{M,L})=0$. This in conjunction with \eqref{TVcond} proves \eqref{TV0}.
\hfill $\square$

\section{Proofs of Theorem \ref{general} and Corollaries \ref{joint}, \ref{pairs} and \ref{densityatpos}}
\noindent \it Proof of Theorem \ref{general}.\rm\
Let $b_n$  satisfy $\lim_{n\to\infty}b_n=b\in[0,1]$ with $b_nn$ an integer,  and let
 $d_n$  satisfy $\lim_{n\to\infty}d_n=d\in[0,1]$, with $d_nn$ an integer.
 As in the proof of Theorem \ref{TV}, we  use the generic $P$ to denote probabilities of events
concerning the shuffling mechanism.
 Let $Q^{(n)}_{b_n,d_n}(x)$, $0\le x\le 1$, denote the rescaled distribution function of $\sigma^{-1}_{b_nn}$ under $p_n(\text{id},\cdot)$, when conditioned on
 card $b_nn$ having been removed and reinserted in position $d_nn$;
 that is $Q^{(n)}_{b_n,d_n}(x)=p_n(\text{id},\sigma^{-1}_{b_nn}\le nx|\text{card}\ b_nn\ \text{was reinserted in position}\ d_nn)$.
 Let $G_b(y)$ be as in the statement of the theorem.
 We will show that
   the distribution $Q^{(n)}_{b_n,d_n}(dx)$ corresponding to the distribution function  $Q^{(n)}_{b_n,d_n}(x)$ converges weakly to the $\delta$-measure
   at $G_b(d)$:
 \begin{equation}\label{conddist}
w-\lim_{n\to\infty}Q^{(n)}_{b_n,d_n}(dx)=\delta_{G_b(d)}.
 \end{equation}
It is easy to check that the function $G_b(y)=e^{(1-y)e^{-b}}-(1-y)e^{1-b}$ is increasing in $y\in[0,1]$. Thus, since
the probability that card $b_nn$ is inserted in a position no larger than $d_n n$ is $d_n$,
it follows that $F_b(G_b(d))=d$; that is, $G_b=F_b^{-1}$.
Thus, to complete the proof of the theorem, we need to prove \eqref{conddist}.

For notational convenience, we will sometimes write $j=b_nn$ and $k=d_nn$.
After card number $j$ is removed and reinserted in  position $k$, a certain number of cards
from among those with numbers less than $j$ (which were removed and reinserted before $j$ was) will be to the left
of newly reinserted card number $j$. Denote this random number of cards by $M$.
Of course then, the other cards to the left of newly reinserted card number $j$
are the cards $j+1,\cdots, j+k-1-M$. These cards are the next to be removed and reinserted.
Let $R$ denote the random number of cards out of these $k-1-M$ cards that  end up to the left of card number $j$.
So now card $j$ is in position $M+R+1$. Now it is the turn of the remaining $n-j-k+M+1$ cards, with numbers from
$j+k-M$ up to $n$,  all of which are to the right
of card number $j$, to be removed and reinserted. Let  $S$ denote the random number of cards  out of these cards that end up to the left of card number $j$.
Then at the end of the shuffle, card number $j$ will be in position $M+R+S+1$.

We will show that as $n\to\infty$,  the distribution of $\frac Mn$ converges to $\delta_{\gamma(b,d)}$,
where
\begin{equation}\label{gamma}
\gamma=\gamma(b,d)=\begin{cases} b-(1-d)(1-e^{-b}),\ & \text{if}\ d\ge 1-(1-b)e^b;\\ d,\ & \text{if}\  d\le 1-(1-b)e^b.\end{cases}
\end{equation}
 We will show that as $n\to\infty$,  the distribution of $\frac Rn$ converges to
$\delta_{1-\gamma-(1-d)e^{d-\gamma}}$. Let $t=t(\gamma,d)=1-\gamma-(1-d)e^{d-\gamma}$. We will
show that as $n\to\infty$,  the distribution of $\frac Sn$ converges to
$\delta_{(\gamma+t)(e^{1-b-d+\gamma}-1)}$. Thus $Q^{(n)}_{b_n,d_n}$, the rescaled distribution of the final position of card $j$, namely,
the distribution of $\frac{M+R+S+1}n$,  will converge to $\delta_{\gamma+t+(\gamma+t)(e^{1-b-d+\gamma}-1)}$.
 Using the equations above to write everything only in terms of $b$ and $d$, we obtain
\begin{equation*}
\gamma+t+(\gamma+t)(e^{1-b-d+\gamma}-1)=\begin{cases}
de^{1-b},\ & \text{if}\ d\le1-(1-b)e^b;\\
e^{(1-d)e^{-b}}-(1-d)e^{1-b},\ & \text{if}\ d\ge 1-(1-b)e^b,\end{cases}
\end{equation*}
thus giving
\eqref{conddist}.

We now prove the claims in the above paragraph regarding the distributions of $\frac Mn$, $\frac Rn$ and $\frac Sn$.
We start with $\frac Mn$.
A careful analysis of the shuffle up until the time that card number $j$ is removed and reinserted in position $k$   will reveal
that if $j\le k$ and  $0\le m\le j-1$, or if $k<j$ and $0\le m\le k-2$,
then the random variable $M$ will be equal to $m$ if and only if
at least $m$ cards from among the first $j-1$ cards were inserted to the left of card number $j+k-m$, and at most
$m$ cards from among the first $j-1$ cards were inserted to the left of card number $j+k-m-1$.
However if   $k<j$ and $m=k-1$,  then the random variable $M$ will be equal to $m=k-1$  if and only
if at least $m=k-1$ out of the first $j-1$ cards were inserted to the left of card number $j+1$.

We can write the probabilities of the events described above   in terms of  certain probabilities for  certain
geometric random variables. For any $i\ge 1$, let $T^{i}_{q_i}$ denote a geometric random variable with parameter $q_i$ and with values in $\{1,2,\cdots\}$,
and let $T^{i}_{q_i}$ and
$T^{j}_{q_j}$ be independent for $j\neq i$.
Let $A_{n,j,k;m}$ denote the event that at least $m$ cards  from among the first $j-1$ cards were inserted to the left of card number $j+k-m$
(with $m$ in the range noted above).
Then
\begin{equation}\label{A}
P(A_{n,j,k;m})=P(\sum_{l=1}^{j-m}T^{l}_{1-\frac{j+k-m-l}n}> j-1).
\end{equation}
The explanation for this is similar to  that given at the point in the proof of Theorem \ref{TV} where geometric random variables were introduced.
(Think of the number of cards that are removed and reinserted until the first time one of them gets placed to the \it right\rm\ of card
number $j+k-m$ as a $T^{1}_{1-\frac{j+k-m-1}n}$ random variable, think of the number of cards that are removed and reinserted after
the first one gets placed to the right of card number $j+k-m$ until a second one gets placed to the right of card number
$j+k-m$ as a $T^{2}_{1-\frac{j+k-m-2}n}$, etc., with the same caveat as noted in the proof
of Theorem \ref{TV}.)

Letting $B_{n,j,k;m}$ denote the event that
at most
$m$ cards from among the first $j-1$ cards were inserted to the left of card number $j+k-m-1$
(with $m$ in the range noted above), we have similarly
\begin{equation}\label{B}
P(B_{n,j,k;m})=P(\sum_{l=1}^{j-m-1}T^{l}_{1-\frac{j+k-m-l}n}\le j-1).
\end{equation}
For $k<j$, letting $C_{n,j,k}$ denote the event that
at least $k-1$ out of the first $j-1$ cards were inserted to the left of card number $j+1$,
we have similarly
\begin{equation}\label{C}
P(C_{n,j,k})=P(\sum_{l=0}^{j-k}T^{l}_{1-\frac{j-l}n}> j-1).
\end{equation}

Recall that $j=b_nn$ and $k=d_nn$. Write
$m$ in the form $m=\gamma_n n$ and assume that $\gamma=\lim_{n\to\infty} \gamma_n$ exists. By the restrictions
on $m$, we can assume that $b\ge\gamma$. Then, by the law of large numbers if $b>\gamma$, and trivially if $b=\gamma$, it follows that
$\frac1n\sum_{l=1}^{j-m}T^{l}_{1-\frac{j+k-m-l}n}$
  converges almost surely to its  limiting mean value.
The mean of the sum is $\frac1n\sum_{l=1}^{(b_n-\gamma_n)n}\frac1{1-b_n-d_n+\gamma_n+\frac ln}$; thus the  limiting mean value
is $\int_0^{b-\gamma}\frac1{1-b-d+\gamma+x}dx=\log\frac{1-d}{1-d-b+\gamma}$.
On the other hand, $\lim_{n\to\infty}\frac{j-1}n=b$.
Thus, we conclude from \eqref{A} that
\begin{equation}\label{Aagain}
\lim_{n\to\infty}P(A_{n,j,k;m})=
\begin{cases} 1,\ \text{if}\  \log\frac{1-d}{1-d-b+\gamma}>b;\\
0, \ \text{if} \ \log\frac{1-d}{1-d-b+\gamma}<b.\end{cases}
\end{equation}
Making the same type of argument for \eqref{B} and \eqref{C},
we obtain
\begin{equation}\label{Bagain}
\lim_{n\to\infty}P(B_{n,j,k;m})=
\begin{cases} 1,\ \text{if}\  \log\frac{1-d}{1-d-b+\gamma}<b;\\
0, \ \text{if} \ \log\frac{1-d}{1-d-b+\gamma}>b,\end{cases}
\end{equation}
and
\begin{equation}\label{Cagain}
\lim_{n\to\infty}P(C_{n,j,k})=
\begin{cases} 1,\ \text{if}\  \log\frac{1-d}{1-b}>b;\\
0, \ \text{if} \ \log\frac{1-d}{1-b}<b.\end{cases}
\end{equation}

Consider first the case that $j\le k$.
Recalling that $M=m=\gamma_n n$ occurs if and only if $A_{n,j,k;m}$ and $B_{n,j,k;m}$ occur, it follows from  \eqref{Aagain} and \eqref{Bagain}
that the distribution of $\frac Mn$ converges to the $\delta$-measure
at the $\gamma$ which solves the equation  $\log\frac{1-d}{1-d-b+\gamma}=b$. The solution is
$\gamma=b-(1-d)(1-e^{-b})$.

Now consider the case  that $k<j$. Note that in this case, $m\le k-1$, which means that necessarily $\gamma\le d$.
First consider the case $0\le m\le k-2$.
Since  $M=m=\gamma_n n$ occurs if and only if
$A_{n,j,k;m}$ and $B_{n,j,k;m}$ occur, we again
conclude that $\frac Mn$ converges to the $\delta$-measure
at  $\gamma=b-(1-d)(1-e^{-b})$, as long as the right hand side is indeed no greater than $d$.
One finds that the right hand side is no greater than $d$ if and only if
$d\ge 1-(1-b)e^b$. If the opposite inequality holds, then we could  conclude by process
of elimination that
$\frac Mn$ converges to the $\delta$-measure at $\gamma=d$.
However, working directly, we recall that $M=k-1$ occurs if and only if
$C_{n,j,k}$  occurs. Solving the inequality
$\log\frac{1-d}{1-b}>b$ gives $d<1-(1-b)e^b$; thus, we conclude  from \eqref{Cagain} that
$\frac Mn$ converges to the $\delta$-measure at $\gamma=d$ if
$d<1-(1-b)e^b$. This completes the proof that the distribution of $\frac Mn$ converges to $\delta_{\gamma(b,d)}$,
where $\gamma(b,d)$ is given by \eqref{gamma}.

Now we turn to the distribution of $\frac Rn$.  Recall that as we begin to implement the random variable $R$,
card number $j$ is in position $k$, to the left of card number $j$ are $M$ cards that have already been removed and reinserted,
as well as $k-M-1$ cards that are now to be removed and reinserted. The random variable $R$ is the number of these
$k-M-1$ cards that end up to the left of card number $j$.
Using  geometric random variables, similar to the case for the random variable $M$, we have
\begin{equation*}
P(\frac Rn\le t|M=\lambda n)=P(\sum_{l=1}^{k-\lambda n-1-tn}T^{l}_{1-\frac{k-l}n}\le k-\lambda n-1).
\end{equation*}
By the law of large numbers if $d>\lambda+t$, and trivially if $d=\lambda+t$, it follows that
 $\frac1n\sum_{l=1}^{k-\lambda n-1-tn}T^{l}_{1-\frac{k-l}n}$ converges almost surely
to its limiting mean value, which is $\int_0^{d-\lambda-t}\frac1{1-d+x}dx=\log \frac{1-\lambda-t}{1-d}$.
On the other hand $\lim_{n\to\infty}\frac{k-\lambda n-1}n=d-\lambda$. Thus, we conclude that
\begin{equation*}
P(\frac Rn\le t|M=\lambda n)=
\begin{cases} 1,\ \text{if}\ \log \frac{1-\lambda-t}{1-d}<d-\lambda;\\
0, \ \text{if}\ \log \frac{1-\lambda-t}{1-d}>d-\lambda.
\end{cases}
\end{equation*}
This proves that the distribution of $\frac Rn$, conditioned on $M=\lambda n$,  converges to the $\delta$-measure
at the $t=t(\lambda,d)$ which solves the equation $\log \frac{1-\lambda-t}{1-d}=d-\lambda$.
The solution is $t(\lambda,d)=1-\lambda-(1-d)e^{d-\lambda}$.
Since the distribution of $\frac Mn$ converges to the $\delta$-measure at $\gamma=\gamma(b,d)$ given in \eqref{gamma},
and since the distribution of $\frac Rn$, conditioned on $\frac Mn=\gamma$, converges to the $\delta$-measure at
$t(\gamma,d)=1-\gamma-(1-d)e^{d-\gamma}$, we conclude that
the distribution of $\frac Rn$ converges to the $\delta$-measure at $t=t(\gamma,d)=1-\gamma-(1-d)e^{d-\gamma}$, with
$\gamma=\gamma(b,d)$.

We now turn  to the distribution of $\frac Sn$.
Recall that as we begin to implement the random variable $S$, card number $j$ is in position $M+R+1$, and there
are $n-j-k+M+1$ cards, all to the right of card number $j$, which need to be removed and reinserted.
The random variable $S$ is the number of these $n-j-k+M+1$ cards that end up to the left of card $j$.
Using geometric random variables again, we have
\begin{equation*}
P(\frac Sn\le v|M=\lambda n, R=\mu n)=P(\sum_{l=1}^{vn+1}T^l_{\lambda +\mu +\frac ln}> n-j-k+\lambda n+1).
\end{equation*}
By the law of large numbers if $v>0$, and trivially if $v=0$, it follows that $\frac1n\sum_{l=1}^{vn+1}T^l_{\lambda +\mu +\frac ln}$
converges almost surely to its limiting mean value, which is $\int_0^v\frac1{\lambda+\mu+x}dx=\log\frac{\lambda+\mu+v}{\lambda+\mu}$.
On the other hand, $\lim_{n\to\infty}\frac{n-j-k+\lambda n+1}n=1-b-d+\lambda$. Thus, we conclude that
\begin{equation*}
P(\frac Sn\le v|M=\lambda n, R=\mu n)=
\begin{cases} 1,\ \text{if}\ \log\frac{\lambda+\mu+v}{\lambda+\mu}>1-b-d+\lambda;\\
0,\ \text{if}\ \log\frac{\lambda+\mu+v}{\lambda+\mu}<1-b-d+\lambda.
\end{cases}
\end{equation*}
This proves that the distribution of
$\frac S n$, conditioned on $M=\lambda n$ and $R=\mu n$, converges  to the $\delta$-measure at the $v=v(\lambda, \mu)$ which solves
the equation $\log\frac{\lambda+\mu+v}{\lambda+\mu}=1-b-d+\lambda$.
The solution is $v=v(\lambda,\mu,b,d)=(\lambda+\mu)(e^{1-b-d+\lambda}-1)$.
Since the distribution of $\frac Mn$ converges to the $\delta$-measure at $\gamma=\gamma(b,d)$,
and since the distribution of $\frac Rn$ converges to the $\delta$-measure at
$t=t(\gamma,d)$,
it follows that the distribution of $\frac Sn$ converges to the $\delta$-measure at $v(\gamma,t,b,d)
=(\gamma+t)(e^{1-b-d+\gamma}-1)$, with $\gamma=\gamma(b,d)$ and $t=t(\gamma,d)$.
\hfill $\square$
\medskip

\noindent \it Proof of Corollary \ref{joint}.\rm\ The proof of Theorem \ref{general} shows that
with regard to the position of a particular card at the end of the shuffle,
the only randomness that remains when $n\to\infty$ is the randomness incurred by
removing and reinserting that particular card, and not the randomness incurred by removing and reinserting other cards.
Furthermore, as is clear intuitively and also from the above proof, a finite number of changes with regard to the positions of other
cards does not change the limiting distribution of the card in question. The corollary follows from these facts.
\hfill $\square$
\medskip

\noindent \it Proof of Corollary \ref{pairs}.\rm\
First we prove part (i).
Since $P(\Sigma^{-1}_{1,b_1}\le \Sigma^{-1}_{2,b_1})=\frac12$, to prove part (i) it suffices  to show that
$\frac{dP(\Sigma^{-1}_{1,b_1}\le \Sigma^{-1}_{2,b_2})}{db_2}\mid_{b_2=b_1}=(1-b_1)e^{b_1}-\frac12$.
We have
\begin{equation}\label{pairsform}
P(\Sigma^{-1}_{1,b_1}\le \Sigma^{-1}_{2,b_2})=\int_{0\le x\le y\le 1}f_{b_1}(x)f_{b_2}(y)dydx=\int_0^1f_{b_1}(x)(1-F_{b_2}(x))dx.
\end{equation}
From the equation $G_b(F_b(x))=x$, we obtain
$$
\frac{d F_b}{d b}(x)=-\frac{\frac{d G_b}{d b}(F_{b}(x))}{G'_{b}(F_{b}(x))}.
$$
Differentiating \eqref{pairsform} with respect to $b_2$ and using the above equation along with the fact that
$f_{b}(x)=\frac1{G'_{b}(F_{b}(x))}$, we have
\begin{equation*}
\frac{dP(\Sigma^{-1}_{1,b_1}\le \Sigma^{-1}_{2,b_2})}{db_2}\mid_{b_2=b_1}=\int_0^1\frac{\frac{d G_{b}}{d b}\mid_{b=b_1}(F_{b_1}(x))}
{\left(G'_{b_1}(F_{b_1}(x))\right)^2}dx.
\end{equation*}
Making the substitution $x=G_{b_1}(y)$ in the above equation, we obtain
\begin{equation}\label{workableequ}
\frac{dP(\Sigma^{-1}_{1,b_1}\le \Sigma^{-1}_{2,b_2})}{db_2}\mid_{b_2=b_1}=\int_0^1\frac{\frac{d G_{b}}{d b}\mid_{b=b_1}(y)}
{G'_{b_1}(y)}dy.
\end{equation}
Recalling the definition of $G_b$ from Theorem \ref{general}, we have
$$
\frac{d G_{b}}{d b}=\begin{cases} -ye^{1-b},\ &  0\le y\le 1-(1-b)e^b;\\
(1-y)\left(e^{1-b}-e^{-b}e^{(1-y)e^{-b}}\right),\ & 1-(1-b)e^b\le y\le 1,\end{cases}
$$
and
$$
G'_b(y)=\begin{cases} e^{1-b},\ & 0\le y\le 1-(1-b)e^b;\\
e^{1-b}-e^{-b}e^{(1-y)e^{-b}},\ & 1-(1-b)e^b\le y\le 1.\end{cases}
$$
Note then that the quotient $\frac{\frac{d G_{b}}{d b}\mid_{b=b_1}(y)}{G'_{b_1}(y)}$
reduces to $-y$ on $0\le y\le 1-(1-b_1)e^{b_1}$, and reduces to $(1-y)$ on
$1-(1-b_1)e^{b_1}\le y\le 0$.
Thus, from \eqref{workableequ}, we obtain
\begin{equation*}
\begin{aligned}
&\frac{dP(\Sigma^{-1}_{1,b_1}\le \Sigma^{-1}_{2,b_2})}{db_2}\mid_{b_2=b_1}=-\int_0^{ 1-(1-b_1)e^{b_1}}ydy+\int_{ 1-(1-b_1)e^{b_1}}^1(1-y)dy=\\
&(1-b_1)e^{b_1}-\frac12.
\end{aligned}
\end{equation*}

Now we prove part (ii). Recall that $f_1(x)\equiv1$. Thus,
$$
P(\Sigma^{-1}_{1,b}\le \Sigma^{-1}_{2,1})=\int_{0\le x\le y\le 1}f_{b}(x)dydx=\int_0^1(1-x)f_b(x)dx=1-E(b),
$$
where $E(b)$ is as in Corollary \ref{expectation}. Furthermore, from that corollary, it follows that $E(b)>\frac12$ for
$b\in(\tilde b,1)$ and $E(b)<\frac12$, for $b\in[0,\tilde b)$, where $\tilde b$ is the unique $b\in[0,1)$ for which $E(b)=\frac12$.
\hfill $\square$
\medskip

\noindent \it Proof of Corollary \ref{densityatpos}.\rm\
Since $h_x(b)=f_b(x)$, the statements regarding $h_0(b)$ and $h_1(b)$ as well as   (iii) and (iv)    follow from Corollary \ref{density} and the definition of $b_x$.
For part (i), we have $f_0(x)=\frac1{G'_0(G^{-1}_0(x))}$, and
$G'_0(y)=e-e^{1-y}$. Note  that  $G_0(y)$ and $x_b$ are the  same function (one of  $y$ and one of  $b$).
Thus, $h_x(0)=f_0(x)=\frac1{e-e^{1-b_x}}=  \frac{e^{b_x-1}}{e^{b_x}-1}$.

The proof of part (ii) requires long, tedious calculations. One begins by differentiating the equation $G_b(F_b(x))=x$
with respect to $b$, thus obtaining $\frac{dF_b(x)}{db}=-\frac{\frac{dG_b}{db}(G^{-1}_b(x))}{G'_b(G^{-1}_b(x))}$.
Differentiating this new equation with respect to $x$, one obtains
\begin{equation}\label{firstderiv}
\frac{df_b(x)}{db}=\frac{\frac{\frac{dG_b}{db}(G^{-1}_b(x))G''_b(G^{-1}_b(x))}{G'_b(G^{-1}_b(x))}-\frac{dG_b'}{db}(G^{-1}_b(x))}
{(G'_b(G_b^{-1}(x)))^2}.
\end{equation}
Using the formulas for $G_b=G_b(y)$ and its derivatives
in the range $1-(1-b)e^b\le y\le 1$, substituting in \eqref{firstderiv} and making a number of cancelations, one obtains
\begin{equation}\label{firstderivagain}
\frac{df_b(x)}{db}=\frac{e^{-b}(e-e^{(1-G^{-1}_b(x))e^{-b}})}{(G'_b(G_b^{-1}(x)))^2},\  0<b<b_x.
\end{equation}
This shows that the density $h_x(b)=f_b(x)$ is increasing for $0<b<b_x$.
Differentiating  \eqref{firstderivagain} with respect to $b$, and again
using the formulas for $G_b=G_b(y)$ and its derivatives
in the range $1-(1-b)e^b\le y\le 1$, and making a lot of cancelations,
one finally arrives at the formula
$$
\frac{d^2f_b(x)}{db^2}=\frac{e^{-3b}(e-e^{(1-G^{-1}_b(x))e^{-b}})^3}{(G'_b(G^{-1}_b(x)))^4}, \ 0<b<b_x.
$$
This shows that the density  $h_x(b)=f_b(x)$ is convex for  $0<b<b_x$.
\hfill $\square$
\section{Proof of Theorem \ref{basic}}
\noindent \it Proof of Theorem \ref{basic}.\rm\
To prove the theorem, we will need to  consider  a related shuffle.
Fix two (not necessarily distinct) permutations $\sigma,\tau\in S_n$.
Start the deck from $\sigma$ and then use $\tau$ in the following manner to remove and randomly
reinsert each card exactly once: for each $j=1,\cdots, n$, the $j$-th card to be removed  and randomly reinserted is
 the card with the number $\tau_j$ on it.
Let $p_n^{\tau}(\sigma,\cdot)$ denote the resulting distribution. (Note that in terms of these shuffles,
we have $p_n(\sigma,\cdot)=p^\sigma(\sigma,\cdot)$; in particular, $p_n(\text{id},\cdot)=p^\text{id}(\text{id},\cdot)$.)
Let $\text{id}^\text{opp}$ denote the permutation in $S_n$ satisfying $\text{id}^\text{opp}_j=n-j+1$, $j=1,\cdots, n$.
Note then that $p_n^{\text{id}^\text{opp}}(\sigma,\text{id})$ is the probability
of ending up with the identity permutation, if one starts from $\sigma$ and  removes and reinserts the cards one by one,
in the order $n, n-1,\cdots, 1$.

There are $n^n$ possible ways to implement the $p_n( \text{id},\cdot)$  card-cyclic to random insertion shuffle
since each of the $n$ cards is removed once and reinserted in one of $n$ positions. The number of ways that result in the permutation
$\sigma$ is thus $n^np_n( \text{id},\sigma)$. By ``undoing'' any such way, we get a one to one correspondence between
the ways of going from id to $\sigma$ using our original  shuffle, which removes and reinserts the cards in the order
$1,2,\cdots, n$,
and the ways of going from $\sigma$ to id using the shuffle which removes and reinserts the cards in the order
$n,n-1,\cdots, 1$.
Thus, we conclude that
\begin{equation}\label{equiv}
p_n(\text{id},\sigma)=p_n^{\text{id}^\text{opp}}( \sigma, \text{id}).
\end{equation}

We will now calculate $p_n^{\text{id}^\text{opp}}(\sigma, \text{id})$.
The cards begin in the order $\sigma$. Card number $n$  is removed first and randomly reinserted, then card  number $n-1$, etc.
There are $n^n$ different ways of implementing this, and we need  to know how many of these ways
will result in
the cards ending up   in the order id. For any such way, we construct a path
$\{W_j\}_{j=1}^n$ as follows. For each $j\in[n]$, let $W_j$ denote the position in which  card number $j$ was inserted.
It is clear that if the cards are to end up in the order id, then we need $W_j\le W_{j+1}$ for all $j$. However
sometimes this is not enough and we will need instead $W_j<W_{j+1}$.
To see when we only need $W_j\le W_{j+1}$ and when we need $W_j<W_{j+1}$, consider the state of the cards
after the cards numbered $n$ down to $n-j+1$ have been reinserted in such a way that they appear in increasing order from left to right.
The current position of  card number $n-j+1$ is by definition $W_{n-j+1}$. To the right of position $W_{n-j+1}$ one finds
all the cards numbered $n$ down to $n-j+2$. If  card number $n-j$ is also to the right of position $W_{n-j+1}$, then when it
is removed and reinserted in a position which we call $W_{n-j}$,  it will find itself to the left of  card number $n-j+1$ if and only if
$W_{n-j}\le W_{n-j+1}$. However, if  card number $n-j$ is to the left of position $W_{n-j+1}$, then when it is removed and reinserted
in a position which we call $W_{n-j}$, it will find itself to the left of card number $n-j+1$ if and only if
$W_{n-j}<W_{n-j+1}$.

Now given $W_{n-j+1}$, in fact we know to which side of $W_{n-j+1}$  card number $n-j$ is to be found. Recall that $I_{n-j}(\sigma)$ is the number
of inversions involving card number $n-j$ and a card with a lower number. Since none of the cards with
a number lower than or equal to   $n-j$ have been moved yet, it follows that these $I_{n-j}(\sigma)$ cards are to the right of  card number $n-j$,
Furthermore, as noted, all of the cards numbered from $n$ down  to $n-j+2$ are in positions to the right of  $W_{n-j+1}$, and  card number $n-j+1$ is in
position $W_{n-j+1}$.
From this is follows that  card number $n-j$ will
 find  itself to the left of position $W_{n-j+1}$
if and only if
$(n-j-1-I_{n-j}(\sigma))+1\le W_{n-j+1}-1$, or equivalently if and only
if
$ n-j-I_{n-j}(\sigma)< W_{n-j+1}$.

So we conclude that in order for the cards to end up in order id, it is necessary and sufficient that
$\{W_{n-j}\}_{j=0}^{n-1}$ satisfy $W_{n-j}\le W_{n-j+1}$, with strict inequality holding if
$ n-j-I_{n-j}(\sigma)< W_{n-j+1}$. By induction starting with $n$ and descending, it follows that
$W_{n-j}\le n-j$, for all $j=0,\cdots, n-1$; in particular, $W_1=1$.

Now define $Y_j=n+1-W_{n-j+1}$, $j=1,\cdots, n$.
We have $Y_j\le Y_{j+1}$.
In terms of $\{Y_j\}_{j=1}^n$, in order for
the cards to end up in order id, it is necessary and sufficient that
$\{Y_j\}_{j=1}^n$ satisfy $Y_j\le Y_{j+1}$, with strict inequality holding if
$Y_j\le j+I_{n-j}\equiv l_j(\sigma)$.
We have thus established a one-to-one correspondence between the number of ways
 of implementing the shuffle according to $p_n^{\text{id}^\text{opp}}(\sigma,\cdot)$
and  ending up with the cards in the order id, and the number of nondecreasing $l(\sigma)$-paths of length $n$.
The number of such paths has been denoted by $N_n(l(\sigma))$; thus
we conclude that $p_n^{\text{id}^\text{opp}}(\sigma,\text{id})=\frac{N_n(l(\sigma))}{n^n}$, and by \eqref{equiv}, we also
have $p_n( \text{id},\sigma)=\frac{N_n(l(\sigma))}{n^n}$.
\hfill $\square$

\section{Proof of Theorem \ref{estimate}}
Since we know that $N_n(l)$ is strictly monotone in $l$, it suffices to show that
$N_n(n-1,\cdots, n-1)=2^{n-1}$ and  that $N_n(1,2,\cdots, n)=\frac1{n+1}\binom{2n}n$.

For $k\in[n]$, there is a one-to-one correspondence between paths $\{Z_i\}_{i=1}^k$ satisfying
$1\le Z_1<Z_2<\cdots<Z_k=n$ and
 solutions $(a_1,\cdots, a_k)$   with positive integral entries to $\sum_{i=1}^ka_i=n$.
The correspondence is given by $a_1=Z_1$ and $a_i=Z_i-Z_{i-1}$, for $i=2,\cdots, k$.
As is well known, the number of such solutions is
$\binom{n-1}{k-1}$ \cite{F}.
Now a path
 $\{Y_i\}_{i=1}^n$ is a  nondecreasing $l$-path of length $n$
 with $l=(n-1,\cdots, n-1)$ if and only if  there exists a $k\in[n]$ such that
 $Y_j=n$ for $j\ge k$ and such that $1\le Y_1<\cdots<Y_k$.
 For any fixed $k$ the number of such paths was just shown to be
$\binom{n-1}{k-1}$.
Thus $N_n(l-1,\cdots, l-1)=\sum_{k=1}^n\binom{n-1}{k-1}=2^{n-1}$.

We claim that for $l=(1,2,\cdots, n-1)$,  there is a one-to-one correspondence between nondecreasing $l$-paths of length $n$
and   Dyck paths of length $2n$. Recall that a Dyck path of length $2n$ is a path
$\{Z_i\}_{i=0}^{2n}$ satisfying  $Z_0=Z_{2n}=0, Z_j\ge0$  and $|Z_j-Z_{j-1}|=1$, for all $j\in[2n]$.
 As is well known the Catalan number $C_n=\frac1{n+1}\binom{2n}n$ gives the number of such Dyck paths \cite{W}.
 It remains to show the correspondence.
 A Dyck path can be represented as a string of $2n$ bits, $n$ of which are labeled $H$ and
 $n$ of which are labeled $T$,  and such that starting to count from the left,
  at no intermediate stage are there fewer $H$'s than $T$'s.
   Let $\{Y_i\}_{i=1}^n$ be a nondecreasing $l$-path of length $n$ corresponding
 to $l=(1,2,\cdots, n-1)$.
 Now we map this path to the Dyck path which begins with $Y_1$ consecutive $H$'s, then has one $T$, then has
 $Y_2-Y_1$ consecutive $H$'s, then one $T$,  then $Y_3-Y_2$ consecutive $H$'s, then one $T$, and continues in this way
 until  ending  with $Y_n-Y_{n-1}$ consecutive $H$'s and
 one $T$.
 Recalling that by  definition, $Y_i\ge i$ and that  $Y_{i+1}$ is allowed to be equal to $Y_i$ whenever $Y_i>i$,
 it is easy to see that this gives the appropriate one-to-one correspondence.
 \hfill $\square$

\end{document}